\documentclass[10pt]{amsart}

\usepackage{amsmath, amsfonts, amsthm, amssymb, graphicx, fullpage, enumerate, float, caption, subcaption, tikz, hyperref, scalerel, boldline, makecell, longtable, array, bbm}
\usetikzlibrary{graphs}  
   
\newtheorem{theorem}{Theorem}[section]         
\newtheorem{lemma}[theorem]{Lemma}
\newtheorem{corollary}[theorem]{Corollary}  
    
\newtheorem{proposition}[theorem]{Proposition}

\theoremstyle{remark}      
\newtheorem{rem}[theorem]{Remark}
   
\newtheorem*{remn}{Remark on notation}
     
\theoremstyle{definition}  
\newtheorem{definition}[theorem]{Definition}

\def\T{\mathbb{T}}
\def\F{\mathbb{F}}           
\def\N{\mathbb{N}}

\def\R{\mathbb{R}}     
\def\Z{\mathbb{Z}}
     
\def\C{\mathbb{C}}  
\def\P{\mathcal{P}}

\def\cQ{\mathcal{Q}}

\def\bsi{\boldsymbol{i}}
\def\bsr{\boldsymbol{r}}

\def\bsc{{\boldsymbol{c}}}

\def\bsx{\boldsymbol{x}}
\def\bsz{\boldsymbol{z}}
\def\bsn{\boldsymbol{n}}
\def\bsm{\boldsymbol{m}}
\def\bss{\boldsymbol{s}}

\def\bszero{\boldsymbol{0}}
\def\bbP{\mathbb{P}}
\def \l{\ell}
\def\gradient{\nabla} 
\def\grad{\gradient}   
\def\bar#1{\overline{#1}} 

\def\norm#1{\|#1\|}

\def\vphi{\varphi}
\def\vepsilon{\varepsilon}

\def\ns{{\operatorname{ns}}}

\def\cont{\text{cont}}

\def\ns{{\operatorname{ns}}}
\def\ord{{\operatorname{ord}}}

\def\bbA{\mathbb{A}} 
\def\bbP{\mathbb{P}}

\begin{document}
\title{The Furstenberg-S\'ark\"ozy theorem for polynomials in one or more prime variables}  
\author{John R. Doyle \qquad Alex Rice}
  
\begin{abstract} We establish upper bounds on the size of the largest subset of $\{1,2,\dots,N\}$  lacking nonzero differences of the form $h(p_1,\dots,p_{\l})$, where $h\in \Z[x_1,\dots,x_{\l}]$ is a fixed polynomial  satisfying appropriate conditions and $p_1,\dots,p_{\l}$ are prime. The bounds are of the same type as the best-known analogs for unrestricted integer inputs, due to Bloom-Maynard and Arala for $\l=1$, and to the authors for $\l \geq 2$.
 
\end{abstract} 

\address{Department of Mathematics, Oklahoma State University, Stillwater, OK 74075}
\email{john.r.doyle@okstate.edu} 
\address{Department of Mathematics, Millsaps College, Jackson, MS 39210}   
\email{riceaj@millsaps.edu}

\maketitle  
\setlength{\parskip}{5pt}   
 
\section{Introduction}  
  
Over the past half century, an array of results using a variety of methods have concerned the existence of certain differences within dense sets of integers. For $X\subseteq \Z$ and $N\in \N$, let $$D(X,N)=\max \left\{|A|: A\subseteq [N], \ (A-A)\cap X \subseteq \{0\} \right\},$$ where $[N]=\{1,2,\dots,N\}$ and $A-A=\{a-b: a,b\in A\}$. In other words, $D(X,N)$ is the threshold such that a subset of $[N]$ with more than $D(X,N)$ elements necessarily contains two distinct elements that differ by an element of $X$, and in particular $D(X,N)=o(N)$ is equivalent to the statement that every set of natural numbers of positive upper density contains a nonzero difference in $X$.

Furstenberg \cite{Furst} and S\'ark\"ozy \cite{Sark1} independently established $D(S,N)=o(N)$, where $S$ is the set of squares, answering a question of Lov\'asz. While Furstenberg used ergodic methods to obtain a qualitative result, S\'ark\"ozy used Fourier analysis to show $D(S,N) \leq N(\log N)^{-1/3+o(1)}$. In the same series of papers, S\'ark\"ozy \cite{Sark3} established $D(\P-1,N)\leq N (\log\log N)^{-2+o(1)}$, where $\P$ is the set of primes, addressing a question of Erd\H{o}s. A substantial literature has developed on refinements, generalizations, and alternative proofs of these results, see for example \cite{Green}, \cite{Slip}, \cite{Lucier},  \cite{Lucier2}, \cite{Ruz}, \cite{LM}, \cite{lipan}, \cite{Lyall}, \cite{Rice}, \cite{taoblog}, \cite{wang}, \cite{mwang}, and \cite{GreenNew}. 

More generally, for $\l \in \N$ and $h\in \Z[x_1,\dots,x_{\l}]$, the following conditions are immediately necessary for $D(h(\Z^{\l}),N) = o(N)$ (resp. $D(h(\P^{\l}),N)=o(N)$), to avoid counterexamples of the form $A=d\N$. 

\begin{definition} For $\l\in \N$, a nonzero polynomial $h\in \Z[x_1,\dots,x_{\l}]$ is \textit{intersective} if for every $d\in \N$, there exists $\bsr= (r_1,\dots,r_{\l})\in \Z^{\l}$ with $d\mid h(\bsr)$. Equivalently, $h$ is intersective if for every prime $p$, there exists $\bsz_p = ((z_p)_1,\dots,(z_p)_{\l}) \in \Z_p^{\l}$ with $h(\bsz_p)=0$, where $\Z_p$ denotes the $p$-adic integers. Further, $h$ is \textit{$\P$-intersective} if $\bsr$ can always be chosen with $(r_i,d)=1$, or equivalently $\bsz_p$ can always be chosen with $(z_p)_i \not\equiv 0 \ (\text{mod }p)$, for all $1\leq i \leq \l$. 
\end{definition} 
\noindent When considering $\l\geq 2$ variables and hoping to quantitatively improve on the univariate setting, examples like $h(x,y)=(x+y)^2$ force one to impose nonsingularity conditions, the core of which is defined below. 

\begin{definition} Suppose $F$ is a field, $\ell \in \N$, and $g\in F[x_1,\dots,x_{\l}]$ is a homogeneous polynomial. We say that $g$ is \textit{smooth} if the vanishing of $g$ defines a smooth hypersurface in $\mathbb{P}^{\ell-1}$ (as opposed to $\bbA^{\l}$). In other words, $g$ is smooth if the system $ g(\bsx)=\frac{\partial g}{\partial x_1}(\bsx)=\cdots=\frac{\partial g}{\partial x_{\l}}(\bsx)=0 $ has no solution besides $x_1=\cdots=x_{\l}=0$ in $\bar{F}^{\l}$. For a general polynomial $h\in F[x_1,\dots,x_{\ell}]$ with  $h=\sum_{i=0}^k h^i$, where $h^i$ is homogeneous of degree $i$ and $h^k\neq 0$, we say $h$ is \textit{Deligne} if the characteristic of $F$ does not divide $k$ and $h^k$ is smooth. When considering polynomials with integer coefficients, we use the terms \textit{smooth} and \textit{Deligne} as defined above by embedding the coefficients in the field of rational numbers. For $\l=1$, all nonconstant polynomials are Deligne.

\end{definition}

\begin{remn} For the remainder of the paper, we take the notational convention that, for a polynomial $h$, $h^i$ denotes the degree-$i$ homogeneous part of $h$, as opposed to $h$ raised to the $i$-th power.
 
\end{remn}

For $k,\l\geq 2$, let $$\mu(k,\l)=\begin{cases} [(k-1)^2+1]^{-1} & \text{if }\l=2 \\ 1/2 &\text{if }\l \geq 3  \end{cases}, \quad \mu'(k,\l)=\begin{cases} [2(k-1)^2+6]^{-1} & \text{if }\l=2 \\ 1/4 &\text{if }\l \geq 3  \end{cases}. $$ The current knowledge of upper bounds for  $D(h(\Z^{\l}),N)$ for $h\in \Z[x_1,\dots,x_{\l}]$ is summarized in Theorem \ref{mainold} below. The general $\l=1$ case is due to Arala, combining efforts of Bloom and Maynard \cite{BloomMaynard} and the second author \cite{ricemax}, each of which built upon work of Pintz, Steiger, and Szemer\'edi \cite{PSS}. The $\l\geq 2$ case is due to the authors \cite{DR}. The term \textit{strongly Deligne} denotes a large subclass of Deligne, intersective polynomials, the precise definition of which we delay to Section \ref{auxsec}, and sufficient conditions for which we list in Theorem \ref{Dcrit}.

\begin{theorem} \label{mainold} If $h\in \Z[x_1,\dots,x_{\l}]$ is strongly Deligne of degree $k\geq 2$, then $$D(h(\Z^{\ell}),N) \ll_h \begin{cases} N(\log N)^{-c\log\log\log N}  & \text{if }\l=1\\ N\exp(-c(\log N)^{\mu(k,\l)}) & \text{if } \l\geq 2  \end{cases},$$ where $c=c(h)>0$.
\end{theorem} 
\noindent The fact that every intersective polynomial is strongly Deligne when $\l=1$ follows from \cite[Lemma 28]{Lucier}, as discussed in Remark \ref{lucsd}, while the remaining criteria below are established in \cite{DR}. 
\begin{theorem} \label{Dcrit} An intersective, Deligne polynomial $h\in \Z[x_1,\dots,x_{\l}]$ of degree $k\geq 2$ is strongly Deligne if it meets any of the following conditions: 

\begin{enumerate}[(i)] \item $\l \neq 2$, \item There exist $a,b\in \Z$ such that $h(a,b)=0$ and the highest and lowest degree parts of $h(x+a,y+b)$ are smooth,  \item  For an irreducible factorization $h=g_1\cdots g_n$ in $\bar{\Z}[x_1,\dots,x_{\l}]$ and all but finitely many $p\in \P$, $g_i$ has coefficients in $\Z_p$ for some $1\leq i \leq n$, \item For all but finitely many $p$, there exists a $p$-adic integer root of $h$ of multiplicity $1$ or $k$. In particular, this condition holds when $k=2$.
\end{enumerate} 
\end{theorem}

Our main results are as follows. Analogous to strongly Deligne, the term \textit{$\P$-Deligne} refers to a large subclass of Deligne, $\P$-intersective polynomials, the precise definition of which is provided in Section \ref{auxsec}.

\begin{theorem} \label{mainPint} If $h\in \Z[x_1,\dots,x_{\l}]$ is $\P$-Deligne of degree $k\geq 2$, then $$D(h(\P^{\ell}),N) \ll_h \begin{cases} N(\log N)^{-c\log\log\log N}  & \text{if }\l=1\\ N\exp(-c(\log N)^{\mu'(k,\l)}) & \text{if } \l\geq 2  \end{cases},$$ where $c=c(h)>0$.
\end{theorem}

\noindent As before, when $\l=1$, all $\P$-intersective polynomials are $\P$-Deligne. The $\l=1$ case of Theorem \ref{mainPint} improves upon the previous bound $D(h(\P),N)\leq N(\log N)^{-\frac{1}{2k-2}+o(1)}$ for $\P$-intersective $h\in \Z[x]$ of degree $k\geq 2$, due to the second author \cite{Rice}.

\begin{rem} Under GRH, the $\l\geq 2$ bounds in Theorem \ref{mainPint} holds with exponent $2\mu'(k,\l)$. Further, unconditional results with exponent between $\mu(k,\l)$ and $\mu'(k,\l)$ may be possible using techniques from \cite{wang}.
\end{rem}

\begin{theorem}\label{Pcrit} A $\P$-intersective, Deligne polynomial $h\in \Z[x_1,\dots,x_{\l}]$ of degree $k\geq 2$ is $\P$-Deligne if it meets any of the following conditions:

\begin{enumerate}[(i)] \item $\l \neq 2$, \item There exist $a,b\in \{-1,1\}$ such that $h(a,b)=0$ and the highest and lowest degree parts of $h(x+a,y+b)$ are smooth,  \item  Let $h=g_1\cdots g_n$ be an irreducible factorization in $\bar{\Z}[x_1,\dots,x_{\l}]$. For all but finitely many $p\in \P$, there exists $1\leq i \leq n$ such that $g_i$ has coefficients in $\Z_p$ and $x_j\nmid g_i$ for all $1\leq j \leq \l$, \item For all but finitely many $p$, there exists a $p$-adic integer root $\bsz_p$ of $h$, satisfying $(z_p)_i \not\equiv 0 \ (\textnormal{mod }p)$ for $1\leq i \leq \l$, of multiplicity $1$ or $k$. In particular, this condition holds when $k=2$.

\end{enumerate}

\end{theorem} 


\begin{rem} The conclusions of Theorems \ref{mainold} and \ref{mainPint} hold under a more general condition. It suffices that $h$ can be written as $h(\bsx_1,\dots,\bsx_s) = h_1(\bsx_1)+\cdots+h_s(\bsx_s)$, where $h_i\in \Z[x_1,\dots,x_{\l_i}]$ is strongly Deligne (resp. $\P$-Deligne) for $1\leq i \leq s$, and $\l_1+\cdots+\l_s = \l$. In this case, the relevant exponential sums factor as products, reducing the problem to the treatment of strongly Deligne (resp. $\P$-Deligne) polynomials in fewer variables. In particular, this includes the diagonal case where $h(x_1,\dots,x_{\l})=h_1(x_1)+\cdots+h_{\l}(x_{\l})$ with $h_1,\dots,h_{\l} \in \Z[x]$ intersective (resp. $\P$-intersective), as treated in \cite{ricemax}. We stick to the conditions as stated in Theorem \ref{mainPint} for ease of exposition.
\end{rem}

As discussed in Section 1.3 in \cite{DR}, known lower bounds for $D(h(\Z^{\l}),N)$ (which also bound $D(h(\P^{\l}),N)$) take the form $N^{1-c}$, so are far removed from known upper bounds, and all are derived from a construction of Ruzsa \cite{Ruz2} (see also \cite{Lewko}, \cite{younis}). Another construction of Ruzsa \cite{Ruz3} shows $D(\P-1,N) \gg N^{c/\log\log N}$, as compared with Green's \cite{GreenNew} breakthrough upper bound $D(\P-1,N) \ll N^{1-c}$.

\section{Preliminaries}

\subsection{Auxiliary polynomials and $\P$-Deligne definition} \label{auxsec}
We employ a standard density increment procedure, which takes as input a set lacking nonzero differences in the image of a polynomial, and produces a new, denser subset of a slightly smaller interval lacking nonzero differences in the image of a potentially modified polynomial. The following definition keeps track of the changes in the polynomial over the course of the iteration.

\begin{definition} Suppose $h\in \Z[x_1,\dots,x_{\ell}]$ is an intersective polynomial and fix, for each prime $p$, $\bsz_p\in \Z_p^{\ell}$ with $h(\bsz_p)=0$. All objects defined below depend on this choice of $p$-adic integer roots, but we suppress that dependence in the subsequent notation. 

\noindent By reducing modulo prime powers and applying the Chinese Remainder Theorem, the choice of roots determines, for each $d\in \N$, a unique $\bsr_d \in (-d,0]^{\ell}$ with $\bsr_d \equiv \bsz_p \ \text{mod }p^j$ for all prime powers $p^j\mid d$.
  
\noindent We define a completely multiplicative function $\lambda$ (depending on $h$ and $\{\bsz_p\}$) on $\N$ by letting $\lambda(p)=p^{m_p}$ for each prime $p$, where $m_p$ is the multiplicity of $\bsz_p$ as a root of $h$, in other words $$m_p=\min\left\{i_1+\cdots+i_{\l} : \frac{\partial^{i_1+\cdots+i_{\l}}h}{\partial x_1^{i_1}\cdots \partial x_{\l}^{i_{\l}}} (\bsz_p) \neq 0\right\}. $$
\noindent For each $d\in \N$, we define the \textit{auxiliary polynomial}, $h_d\in \Z[x_1,\dots,x_{\l}]$, by $$h_d(\bsx)=h(\bsr_d+d\bsx)/\lambda(d). $$
\end{definition}

\begin{definition} We say that $h$ is \textit{strongly Deligne} if there exists a finite set of primes $X=X(h)$ and a choice $\{\bsz_p\}_{p\in \P}$ of $p$-adic integer roots of $h$ such that the reduction of $h_d$ modulo $p$ is Deligne for all $p\notin X$ and all $d\in \N$. Analogously, we say $h$ is  \textit{$\P$-Deligne} if such a choice  $\{\bsz_p\}_{p\in \P}$ exists with $(z_p)_i \not\equiv 0 \ (\text{mod }p)$ for all $p$ and all $1\leq i \leq \l$. We note that all strongly Deligne polynomials are both Deligne and intersective, and all $\P$-Deligne polynomials are both Deligne and $\P$-intersective. 
\end{definition} 

\begin{rem}\label{lucsd} For $\l=1$, all nonconstant polynomials are Deligne, so it follows from \cite[Lemma 28]{Lucier} that all intersective polynomials are strongly Deligne and all $\P$-intersective polynomials are  $\P$-Deligne, where $X$ is the set of primes that could simultaneously divide all nonconstant coefficients of an auxiliary polynomial. 
\end{rem}

\subsection{Counting Primes in Arithmetic Progressions}
 
For $x>0$ and $a,q\in \N$, we define 
\begin{equation*} \psi(x,a,q) = \sum_{\substack{p \leq x \ \text{prime} \\ p\equiv a (\text{mod } q)}} \log p.
\end{equation*} 
Classical estimates on $\psi(x,a,q)$ come from the Siegel-Walfisz Theorem, which can be found for example in Corollary 11.19 of \cite{MV}. Ruzsa and Sanders \cite{Ruz} established asymptotics for $\psi(x,a,q)$ for certain moduli $q$ beyond the usual limitation $q\leq (\log x)^C$ by exploiting a dichotomy based on exceptional zeros, or lack thereof, of Dirichlet $L$-functions. In particular, the following result follows from their work.

\begin{lemma} \label{RS} For any $Q,D>0$, there exist $q_0 \leq Q^{D}$ and $\rho \in [1/2,1)$ with $(1-\rho)^{-1} \ll q_0$ such that 
\begin{equation}\label{lb} \psi(x,a,q)= \frac{x}{\vphi(q)}-\frac{\chi(a)x^{\rho}}{\vphi(q)\rho}+O\left(x \exp\left(-\frac{c\log x}{\sqrt{\log x}+D^2\log Q}\right)D^2\log Q\right),
\end{equation}
where $\chi$ is a Dirichlet character modulo $q_0$, provided $q_0 \mid q$, $(a,q)=1$,  and $q \leq (q_0Q)^{D}$.
\end{lemma}  

\noindent Lemma \ref{RS} is a purpose-built special case of \cite[Proposition 4.7]{Ruz}, which in the language of that paper can be deduced by considering the pair $(Q^{D^2+D},Q^{D})$, where $q_0$ is the modulus of the exceptional Dirichlet character if the pair is exceptional and $q_0=1$ if the pair is unexceptional. \\

\noindent It is a calculus exercise to verify that if $\epsilon \in [0,1/2]$ and $x\geq 16$, then $1-x^{-\epsilon}/(1-\epsilon)\geq \epsilon,$ which implies that the main term in Lemma \ref{RS} satisfies 
\begin{equation}\label{x/q0}\Re\left(\frac{x}{\vphi(q)}-\frac{\chi(a)x^{\rho}}{\vphi(q)}\right) \geq (1-\rho) \frac{x}{\vphi (q)} \gg \frac{x}{q_0\vphi(q)}.
\end{equation}

\subsection{Fourier analysis and the circle method on $\Z$} We embed our finite sets in $\Z$, on which we utilize an unnormalized discrete Fourier transform. Specifically, for a function $F: \Z \to \C$ with finite support, we define $\widehat{F}: \T \to \C$, where $\T= \R/\Z$, by \begin{equation*} \widehat{F}(\alpha) = \sum_{x \in \Z} F(x)e(-x\alpha), \end{equation*} where $e(t)=e^{2\pi i t}$. Given $N\in \N$ and a set $A\subseteq [N]$ with $|A|=\delta N$, we examine the Fourier analytic behavior of $A$ by considering the \textit{balanced function}, $f_A$, defined by
$ f_A=1_A-\delta 1_{[N]}.$

\noindent As is standard in the circle method, we decompose the frequency space into two pieces: the points of $\T$ that are close to rational numbers with small denominator, and the complement.

\begin{definition}Given $\gamma>0$ and $Q\geq 1$, we define, for each $a,q\in \N$ with $0\leq a \leq q-1$,
$$\mathbf{M}_{a/q}(\gamma)=\left\{ \alpha \in \T : \Big|\alpha-\frac{a}{q}\Big| < \gamma \right\},  \ \mathbf{M}_q(\gamma)=\bigcup_{(a,q)=1} \mathbf{M}_{a/q}(\gamma), \text{ and }\mathbf{M}'_q(\gamma)=\bigcup_{r\mid q} \mathbf{M}_r(\gamma)=\bigcup_{a=0}^{q-1} \mathbf{M}_{a/q}(\gamma).$$  We then define the \textit{major arcs} by
$$ \mathfrak{M}(\gamma,Q)=\bigcup_{q=1}^{Q} \mathbf{M}_q(\gamma),$$ 
and the \textit{minor arcs} by  
$\mathfrak{m}(\gamma,Q)=\T\setminus \mathfrak{M}(\gamma,Q).
$ 
We note that if $2\gamma Q^2<1$, then \begin{equation} \label{majdisj}\mathbf{M}_{a/q}(\gamma)\cap\mathbf{M}_{b/r}(\gamma)=\emptyset \end{equation}whenever  $a/q\neq b/r$ and  $q,r \leq Q$. 
\end{definition}

\subsection{Inheritance proposition} As previously noted, we defined auxiliary polynomials to keep track of an inherited lack of prescribed differences at each step of a density increment iteration. For our restriction to prime inputs, we define \begin{equation*} \Lambda_d = \{\bsn\in \Z^{\l}: \bsr_d+d\bsn \in \P^{\l} \}.\end{equation*} The following proposition makes this inheritance precise.

\begin{proposition} \label{inh} If $h\in\Z[x_1,\dots,x_{\ell}]$ is $\P$-intersective, $d,q\in \N$, $A\subseteq \N$, and $A'\subseteq \{a: x+\lambda(q)a \in A\}$ for some $x\in \Z$, then $$\{\lambda(q)x: x\in (A'-A')\cap h_{qd}(\Lambda_{qd})\}\subseteq  (A-A)\cap h_d(\Lambda_d).$$

\noindent In particular, if $(A-A)\cap h_d(\Lambda_d)\subseteq \{0\}$, then $(A'-A')\cap h_{qd}(\Lambda_{qd})\subseteq \{0\}$.
\end{proposition}
\begin{proof} Suppose that $A\subseteq \N$, $A'\subseteq \{a : x+\lambda(q)a \in A\}$, and $t\in(A'-A')\cap h_{qd}(\Lambda_{qd})$. In other words,  $$t = a-a'=h_{qd}(\bsn)=h(\bsr_{qd}+qd\bsn)/\lambda(qd)$$ for some $\bsn\in \Z^{\ell}$, $a, a' \in A'$, with $\bsr_{qd}+qd\bsn \in \P^{\l}$. Clearly $\lambda(q)t=(x+\lambda(q)a)-(x+\lambda(q)a')\in A-A$, but we must also show $\lambda(q)t\in h_d(\Lambda_d)$. By construction, we have that $\bsr_{qd}\equiv \bsr_d$ mod $d$, so there exists $\boldsymbol{s}\in \Z^{\ell}$ such that $\bsr_{qd}=\bsr_d+d\boldsymbol{s}$. Further, $\lambda$ is completely multiplicative, and therefore $$\lambda(q)t = \lambda(q)h_{qd}(\bsn) = h(\bsr_d+d(\boldsymbol{s}+q\bsn))/\lambda(d) = h_d(\boldsymbol{s}+q\bsn),$$ and $\bsr_d+d(\boldsymbol{s}+q\bsn) = \bsr_{qd}+qd\bsn \in \P^{\l}$, so $\lambda(q)t\in h_d(\Lambda_d)$, completing the proof.
\end{proof} 

\subsection{Input restriction} \label{sievesec} 
 
In addition to restricting our inputs to (affine preimages of) the primes, we also sieve away roots of the gradient vector, as is done in \cite{ricemax} and \cite{DR}, in order to apply Hensel's lemma when analyzing local exponential sums. For a $\P$-intersective polynomial $h\in\Z[x_1,\dots,x_{\ell}]$, and each prime $p$ and $d\in \N$, we define $\vepsilon_d(p)$ to be $0$ if $p\mid d$ and $1$ otherwise. We define $\gamma_{d}(p)$ to be the smallest power such that $\grad h_d$ does not vanish modulo $p^{\gamma_d(p)}$ on \begin{equation*} J_d(p)=\{\bsc \in (\Z / p^{\gamma_{d}(p)}\Z)^{\l} : (r_d)_i+dc_i \not\equiv 0 \ (\text{mod }p) \text{ for all }1\leq i \leq \l\}. \end{equation*} We let $j_d(p)$ denote the number of solutions to $\grad h_d\equiv \bszero \ (\text{mod }p^{\gamma_d(p)})$ on $J_d(p)$, and we let $$w_d=\prod_{p\leq Y} \left(1-\frac{j_d(p)}{|J_d(p)|}\right)=\prod_{p\leq Y} \left(1-\frac{j_d(p)}{\left((p-\vepsilon_d(p))p^{\gamma_d(p)-1}\right)^{\l}}\right)>0.$$  Then, for $d\in \N$ and $Y>0$ we define $$W_d(Y)=\left\{\bsn\in \Lambda_d: \grad h_d(\bsn) \not\equiv \bszero \text{ mod } p^{\gamma_{d}(p)} \text{ for all } p\leq Y \right\}.$$

\noindent To appropriately weight our doubly-restricted inputs, we define \begin{equation*}\nu_d(\bsn) = w_d^{-1}\cdot \left( \frac{\vphi(d)}{d}\right)^{\l} \cdot \prod_{i=1}^{\l}\log((r_d)_i+dn_i) \cdot 1_{W_d(Y)}(\bsn) \cdot \begin{cases} h_d'(n) & \l=1 \\ 1 & \l\geq 2 \end{cases}. \end{equation*}

\noindent Further, for $x_1,\dots,x_{\l}>0$ and a collection of primes $p_1<\cdots<p_{s}$, we define 
\begin{equation*}\mathcal{A}_{p_1\cdots p_s}=\mathcal{A}_{p_1\cdots p_s}(x_1,\dots,x_{\l}) = \left\{\bsn \in B: \bsr_d+d\bsn \in \P^{\l}, \ \grad h_d (\bsn) \equiv \bszero\text{ mod } p_i^{\gamma(p_i)} \ \text{for all } 1\leq i \leq s  \right\},
\end{equation*}
where $B=[x_1]\times\cdots\times [x_{\l}]$. Finally, we let $b=b(h)=\max_{p\in \P} \gamma_d(p) \ll_h 1$ (see Section \ref{cdgv}).
  

\noindent For the three lemmas that follow, we let $Q,D,q_0,\rho,$ and $\chi$ be as in Lemma \ref{RS}.

\begin{lemma} \label{sieve2} If $\l \geq 2$, $x_1,\dots,x_{\l},Y>0$, $q_0\mid d$, $dY^{bt}\leq (q_0Q)^D$, and $t>2C\log\log Y$ for a constant $C=C(h)$, then \begin{equation*} \sum_{\bsn \in B} \nu_d(\bsn) =\prod_{i=1}^{\l} \left(x_i-\frac{\chi((r_d)_i)x_i^{\rho}}{\rho d^{1-\rho}} \right) + E_1 + E_2,
\end{equation*}
where $X=\max\{x_1,\dots,x_{\l}\}$, $E_1= O_h\left((dX)^{\l}\Big(\frac{C\log\log Y}{t}\Big)^t \right),$ and $$E_2=O_h\left((dX)^{\l}\log^{\l}(dX)Y^{4bt}\exp\left(-\frac{c\log dx}{\sqrt{\log dx}+D^2\log Q}\right)D^{2\l}\log^{\l} Q\right), $$ for a constant $c>0$.
\end{lemma}
 
\begin{proof} Suppose $\l\geq 2$, and fix $x_1,\dots,x_{\l},Y>0$ and $d\in \N$ with $q_0 \mid d$. Letting $z$ denote the number of primes that are at most $Y$, we have by the Chinese Remainder Theorem and the inclusion-exclusion principle that \begin{equation*}\label{brunalt2} \sum_{\bsn \in B} \nu_d(\bsn)= w_d^{-1} \left(\frac{\vphi(d)}{d}\right)^{\l}\sum_{s=0}^z  (-1)^s \sum_{p_1<\dots < p_s\leq Y} \sum_{\mathcal{A}_{p_1\cdots p_s}} \prod_{i=1}^{\l} \log((r_d)_i+dn_i), \end{equation*}
and moreover the true sum lies between any two consecutive truncated alternating sums in $s$. 

\noindent Note that for an $\l$-tuple of congruence classes $\bsc$ modulo $m=p_1^{\gamma_d(p_1)}\cdots p_s^{\gamma_d(p_s)}$ to contribute more than a single element to $\mathcal{A}_{p_1\cdots p_s}(x)$, we must have $p_i \nmid (r_d)_j+dc_j$ for all $1\leq i \leq s$ and all $1\leq j \leq \l$. We enumerate these congruence classes $\bsc_i$ for $1\leq i \leq j_d(p_1)\cdots j_d(p_s),$ and letting $b=\max_{p\in\P} \gamma_d(p)$, we know by Lemma \ref{wprimes} that if $dY^{bt}\leq (q_0Q)^D$, then
\begin{align*} &w_d^{-1}\left(\frac{\vphi(d)}{d}\right)^{\l}\sum_{s=0}^{t}  (-1)^s \sum_{p_1<\dots < p_s \leq Y} \left( \sum_{i=1}^{j_d(p_1)\cdots j_d(p_s)} \sum_{\substack{\bsn\in B\cap \Lambda_d \\ \bsn\equiv \bsc_i  (\text{mod }m)}} \prod_{i=1}^{\l}\log((r_d)_i+dn_i) + O_h(Y^{bt}\log^{\l}(dX)) \right) \\ = & w_d^{-1}\left(\frac{\vphi(d)}{d}\right)^{\l}\sum_{s=0}^{t}  (-1)^s \sum_{p_1<\dots < p_s \leq Y} \left( \sum_{i=1}^{j_d(p_1)\cdots j_d(p_s)} \prod_{i=1}^{\l}\psi((r_d)_i+dx_i,(r_d)_i+dc_i,dm) + O_h(Y^{bt}\log^{\l}(dX)) \right) \\ = & w_d^{-1}\left(\frac{\vphi(d)}{d}\right)^{\l}\sum_{s=0}^{t}  (-1)^s \sum_{p_1<\dots < p_s \leq Y} \left( \frac{j_d(p_1)\cdots j_d(p_s)}{(\vphi(dm))^{\l}} \prod_{i=1}^{\l} \left(dx_i-\chi((r_d)_i)(dx_i)^{\rho}/\rho\right)  + C^sE_0\right)\\ = & w_d^{-1}\sum_{s=0}^{t}  (-1)^s \sum_{p_1<\dots < p_s \leq Y} \left( \prod_{i=1}^{\l} \frac{j_d(p_i)}{\left((p-\epsilon_d(p_i))p^{\gamma_d(p_i)-1}\right)^{\l}}\left(x_i-\frac{\chi((r_d)_i)x_i^{\rho}}{\rho d^{1-\rho}} \right)  + C^sE_0 \right),  \end{align*}
where the $C^s$ term comes from the fact that $j_d(p)\ll_h 1$ (see Section \ref{cdgv}), and \begin{equation*}E_0=O_h\left(Y^{bt}\log^{\l}(dX)+(dX)^{\l}\exp\left(-\frac{c\log dx}{\sqrt{\log dx}+D^2\log Q}\right)D^{2\l}\log^{\l} Q\right), \end{equation*}
which can be merged to \begin{equation} \label{E0b1} E_0=O_h\left((dX)^{\l}\log^{\l}(dX)Y^{bt}\exp\left(-\frac{c\log dx}{\sqrt{\log dx}+D^2\log Q}\right)D^{2\l}\log^{\l} Q\right). \end{equation}
Adding and subtracting the main term for $s>t$, and noting that $$\sum_{s=0}^{z}  (-1)^s \sum_{p_1<\dots < p_s \leq Y} \prod_{i=1}^s \frac{j_d(p_i)}{\left((p-\epsilon_d(p_i))p^{\gamma_d(p_i)-1}\right)^{\l}}=w_d, $$  we have 
\begin{equation} \label{E1E21} \sum_{\bsn\in B} \nu_d(\bsn) =  \prod_{i=1}^{\l} \left(x_i-\frac{\chi((r_d)_i)x_i^{\rho}}{\rho d^{1-\rho}} \right)+E_1+E_2,  
\end{equation} 
where $E_1=O\left((dX)^{\l}\sum_{s=t+1}^{z}\sum_{p_1<\dots < p_s \leq Y}  \prod_{i=1}^s \frac{j_d(p_i)}{\left((p-\epsilon_d(p_i))p^{\gamma_d(p_i)-1}\right)^{\l}} \right)$ and $E_2 = \sum_{s=0}^t {z\choose s}C^sE_0$, and we note that $w^{-1}=O_h(1)$ since $\l\geq 2$. First, we see 
\begin{align*} E_1&=O_h\left((dX)^{\l}\sum_{s=t+1}^{z}\sum_{p_1<\dots < p_s \leq Y}  \prod_{i=1}^s \frac{C}{p} \right)\\& =O_h\left((dX)^{\l}\sum_{s=t+1}^{z}\frac{1}{s!}\Big(\sum_{p\leq Y}\frac{C}{p}\Big)^s \right) \\ & = O_h\left((dX)^{\l}\sum_{s=t+1}^{z}\frac{(C\log\log Y)^s}{s!} \right).
\end{align*}
If $t>2C\log\log Y$, then the sum is dominated by twice its first term, hence \begin{equation}\label{E1b1} E_1=O_h\left((dX)^{\l}\frac{(C\log\log Y)^t}{t!} \right) = O_h\left((dX)^{\l}\Big(\frac{C\log\log Y}{t}\Big)^t \right), \end{equation} where the last step uses that $t!\geq (t/e)^t$. 

\noindent Further,
\begin{equation*}|E_2| = \left|\sum_{s=0}^t {z\choose s}C^sE_0\right| \leq |E_0|\sum_{s=0}^t \frac{(Cz)^s}{s!}.
\end{equation*} 
This time the sum is dominated by twice its \textit{last} term, so \begin{equation} \label{E2b1} E_2 = O(E_0(Cz)^t/t!)=O(E_0(Cz/t)^t). \end{equation} Since $z\leq Y$, the result follows from \eqref{E1E21}, \eqref{E0b1}, \eqref{E1b1}, and \eqref{E2b1}.
 \end{proof} 

\noindent In particular, by the same reasoning as \eqref{x/q0}, the main term in Lemma \ref{sieve2} satisfies \begin{equation} \Re\left( \prod_{i=1}^{\l} \left(x_i-\frac{\chi((r_d)_i)x_i^{\rho}}{\rho d^{1-\rho}} \right)\right) \gg x_1\cdots x_{\l}/q_0^{\l}. \end{equation}

\noindent To account for the derivative weight when $\l=1$, we need the following estimate, which follows from Lemma \ref{RS} and partial summation.

\begin{lemma} \label{wprimes} If $x>0$, $q_0\mid q$, $(a,q)=1$, and $q\leq (q_0Q)^D$, then \begin{align*}\Psi_h(x,a,q) = \sum_{\substack{p\leq x \textnormal{ prime} \\ p\equiv a (\textnormal{mod }q)}} h'(p)\log p  = & \vphi(q)^{-1}\int_1^x h'(t)\left(1-\chi(a)t^{\rho-1}\right) dt\\ & + O_h\left(x^k\exp\left(-\frac{c\log x}{\sqrt{\log x}+D^2\log Q}\right)D^2\log Q\right). \end{align*} 
\end{lemma}
 
\noindent Incorporating Lemma \ref{wprimes}, we have the following analog of Lemma \ref{sieve2} when $\l=1$.

\begin{lemma} \label{sieve1} If $\l=1$, $x,Y>0$, $q_0\mid d$, $dY^{bt}\leq (q_0Q)^D$, and $t>2C\log\log Y$ for a constant $C=C(h)$, then \begin{equation*}\sum_{n\leq x} \nu_d(n) = \int_1^x h_d'(t)\left(1-\chi(r_d)(dt)^{\rho-1}\right)dt+E_1+E_2, \end{equation*} where $$E_1= O_h\left(h_d(x)(\log Y)^k\Big(\frac{C\log\log Y}{t}\Big)^t \right),$$ and $$E_2=O_h\left((dx)^k\log(dx)Y^{4bt}\exp\left(-\frac{c\log dx}{\sqrt{\log dx}+D^2\log Q}\right)D^2\log Q\right), $$ for a constant $c>0$.
\end{lemma}

\section{Proof of Theorem \ref{mainPint}}

\subsection{Main iteration lemma and proof of Theorem \ref{mainPint}}

For the remainder of the section, we fix a $\P$-Deligne polynomimal $h\in \Z[x_1,\dots,x_{\ell}]$ of degree $k\geq 2$,  and positive constants $C_0=C_0(h)$ and $c_0=c_0(h)$ that are appropriately large and small, respectively. We also fix $N\in \N$, and we let $$\cQ=\cQ(N,h)=\exp(c_0\sqrt{\log N}), \quad \cQ'=\cQ'(N,h)=\exp(c_0(\log N)^{1/4}). $$ We apply Lemma \ref{RS} with $Q=\cQ$ and $D=C_0$, letting $q_0\leq \cQ^{C_0}$, $\rho\in [1/2,1)$ and the Dirichlet character $\chi$ be as in the conclusion of the that lemma. For a density $\delta\in (0,1]$, we define $$\theta(h,\delta) = \begin{cases} \exp\left(-C_0\frac{\log(2\delta^{-1})}{\log\log(3\delta^{-1})} \right) & \l = 1 \\ \log^{-[(k-1)^2+1]}(2\delta^{-1}) & \l=2 \\ 1 & \l\geq 3 \end{cases}.$$

We deduce Theorem \ref{mainPint} from the following iteration lemma, encapsulating the density increment strategy.

\begin{lemma} \label{mainit} Suppose $A\subseteq [L]$ with $|A|=\delta L$ and $L\geq \sqrt{N}$. If $(A-A)\cap h_d(\Lambda_d)\subseteq \{0\}$,  $C_0,\delta^{-1}\leq \cQ'$, $q_0 \mid d$, and $d/q_0\leq \cQ$,  then there exists $q\ll_{h} \delta^{-2}$ and $A'\subseteq [L']$ such that 
$L'\gg_{h}  \delta^{4k}L$, 
\begin{equation*}|A'|\geq (1+c\theta(h,\delta))\delta L'   , \end{equation*} where $c=c(h)>0$, and $$(A'-A')\cap h_{qd}(\Lambda_{qd})\subseteq \{0\}. $$

\end{lemma} 

\begin{proof}[Proof of Theorem \ref{mainPint}] Throughout this proof, we let $C$ and $c$ denote sufficiently large or small positive constants, respectively, which we allow to change from line to line, but which depend only on  $h$.  

\noindent Suppose $A \subseteq [N]$ with $|A|=\delta N$ and $$(A-A)\cap h(\P^{\ell})\subseteq \{0\}.$$ Partitioning $[N]$ into arithmetic progressions of step size $\lambda(q_0)$ and length between $N/2\lambda(q_0)$ and $N/\lambda(q_0)$, the pigeonhole principle guarantees the existence of an arithmetic progression $P=\{x+a\lambda(q_0): 1\leq \ell \leq N_0\}$ such that $N/2\lambda(q_0) \leq N_0 \leq N/\lambda(q_0)$ and $|A\cap P|/N_0 = \delta_0 \geq \delta$. This allows us to define $A_0 \subseteq [N_0]$ by $A_0=\{a \in [N_0]: x+a\lambda(q_0) \in A \}$. Setting $d_0=q_0$, Lemma \ref{mainit} yields, for each $m$, a set $A_m \subseteq [N_m]$ with $|A_m|=\delta_mN_m$ and $(A_m-A_m)\cap h_{d_m}(\Lambda_{d_m})\subseteq \{0\}.$ Further, we have that
\begin{equation} \label{NmI} N_m \geq c\delta^{4k}N_{m-1} \geq (c\delta)^{4km} N,
\end{equation}
\begin{equation} \label{incsizeI} \delta_m \geq (1+c\theta(h,\delta_{m-1}))\delta_{m-1}, 
\end{equation}
and
\begin{equation}\label{dmI} d_m \leq (c\delta)^{-2} d_{m-1} \leq (c\delta)^{-2 m}, 
\end{equation}
as long as 
\begin{equation} \label{delmI} C,  \delta_m^{-1}\leq \cQ', \qquad d_m/q_0 \leq \cQ, \qquad \text{and} \qquad N_m\geq \sqrt{N}.
\end{equation}
By (\ref{incsizeI}), the density $\delta_m$ will exceed $1$, and hence (\ref{delmI}) must fail, for $m=M=M(h,\delta)$, where $$M(h,\delta)=\begin{cases} C\log (2\delta^{-1})& \text{if }\ell \geq 3 \\ C\log^{(k-1)^2+2}(2\delta^{-1}) & \text{if }\ell=2 \\ \exp\left(C\frac{\log(2\delta^{-1})}{\log\log(3\delta^{-1})}\right) & \text{if }\l=1 \end{cases}. $$ However, for $\l\geq 2$, by (\ref{NmI}), (\ref{incsizeI}), and (\ref{dmI}), (\ref{delmI}) holds for $m=M$ if \begin{equation}\label{endgame} (c\delta)^{4kM}=\exp\left(C\log^{[2\mu'(k,\l)]^{-1}}(2\delta^{-1})\right)\leq \cQ = \exp(c\sqrt{\log N}).\end{equation} Therefore, (\ref{endgame}) must fail, or in other words $\delta \ll_{h} \exp(-c(\log N)^{\mu'(k,\l)}),$ as claimed. 

\noindent Similarly, for $\l=1$, (\ref{delmI}) holds for $m=M$ if \begin{equation}\label{endgame2} (c\delta)^{4kM}=\exp\left(\exp\left(C\frac{\log(2\delta^{-1})}{\log\log(3\delta^{-1})}\right)\right)\leq \cQ = \exp(c\sqrt{\log N}).\end{equation} Therefore, (\ref{endgame2}) must fail, which yields \begin{equation*}\frac{\log(2\delta^{-1})}{\log\log(3\delta^{-1})} \geq c\log \log N, \end{equation*} and finally \begin{equation*} \delta \ll_h (\log N)^{-c\log\log\log N}, \end{equation*} completing the proof.
\end{proof}

\subsection{$L^2$ concentration and density increment lemmas} As usual, we prove Lemma \ref{mainit} by locating one small denominator $q$ such that $\widehat{f_A}$ has $L^2$ concentration around rationals with denominator $q$, then invoke a standard lemma stating that $L^2$ concentration of  $\widehat{f_A}$ implies the existence a long arithmetic progression on which $A$ has increased density.

\begin{lemma}  \label{L2I} Suppose $A\subseteq [L]$ with $|A|=\delta L$, $L\geq \sqrt{N}$, $\eta=c_0\delta$, and  $\gamma=\eta^{-2k}/L$. Further suppose $(A-A)\cap h_d(\Lambda_d)\subseteq \{0\}$,  $C_0,\delta^{-1}\leq \cQ'$, $q_0 \mid d$, and $d/q_0\leq \cQ$. If $|A\cap(L/9,8L/9)|\geq 3\delta L/4$, then there exists $q\leq \eta^{-2}$ such that 
\begin{equation*} \int_{\mathbf{M}'_q(\gamma)} |\widehat{f_A}(\alpha)|^2d\alpha \gg_{h} \theta(h,\delta) \delta^{2} L.  
\end{equation*}  
\end{lemma}

\noindent Lemma \ref{mainit} follows from Lemma \ref{L2I} and the following standard $L^2$ density increment lemma.

\begin{lemma}[Lemma 2.3 in \cite{thesis}, see also \cite{Lucier}, \cite{Ruz}] \label{dinc} Suppose $A \subseteq [L]$ with $|A|=\delta L$. If  $0< \theta \leq 1$, $q \in \N$, $\gamma>0$, and
\begin{equation*} \int_{\mathbf{M}'_q(\gamma)}|\widehat{f_A}(\alpha)|^2d\alpha \geq \theta\delta^2 L,
\end{equation*} 
then there exists an arithmetic progression $P=\{x+a q : 1\leq a \leq L'\}$ with $qL' \gg \min\{\theta L, \gamma^{-1}\}  $ and $|A\cap P| \geq (1+\theta/32)\delta L'$.
\end{lemma}
\noindent The deduction of Lemma \ref{mainit} from Lemmas \ref{L2I} and \ref{dinc} is standard, and in particular is identical to the analogous deduction in \cite[Section 4]{DR}.



\subsection{Proof of Lemma \ref{L2I}} \label{L2Isec} Before delving into the proof of Lemma \ref{L2I}, we take the opportunity to define some relevant sets and quantities, depending on our $\P$-Deligne polynomial $h\in \Z[x_1,\dots,x_{\ell}]$, scaling parameter $d$, a parameter $Y>0$, and the size $L$ of the ambient interval.

We define $W_d$, $w_d$, $\gamma_d$, $j_d$, and $\nu_d$ in terms of $h$ as in Section \ref{sievesec}. We  let $M=\left(\frac{L}{9J}\right)^{1/k}$,  where $J$ is the sum of the absolute value of all the coefficients of $h_d$, and hence $h_d([M]^{\ell})\subseteq [-L/9,L/9]$.  We let $Z=\{\bsn\in \Z^{\ell}: h_d(\bsn)=0\}$, and we let $H=\left([M]^{\ell}\cap W_{d}(Y)\right)\setminus Z$. We note that the hypotheses $\cQ'\geq C_0$ and $L\geq \sqrt{N}$ allow us to assume at any point that $\cQ',\cQ,$ and $L$ are all sufficiently large with respect to $h$. Under this assumption, it follows from Lemmas \ref{sieve2} and \ref{sieve1} (with $t=c\sqrt{\log M}/\log Y$), and the estimate 
\begin{equation} \label{Zest} |Z\cap[M]^{\ell}|\ll_h M^{\ell-1},\end{equation}
that for $\l\geq 2$ we have 
\begin{equation*}\label{Hsize} T = \sum_{\bsn \in H} \nu_d(\bsn) \gg_h \left|\prod_{i=1}^{\l}\left(M-\chi((r_d)_i)M^{\rho}/\rho \right) \right| \gg (M/q_0)^{\l}, \end{equation*}
and for $\l=1$ we have \begin{equation*}\label{Hsize1} T = \sum_{n \in H} \nu_d(n) \gg_h \int_1^M h_d'(x)\left(1-\chi(r_d)(dx)^{\rho-1}\right)dx \gg L/q_0. \end{equation*}

\begin{proof}[Proof of Lemma \ref{L2I} for $\l\geq 2$]
\label{massproof} Suppose $\l\geq 2$, $A\subseteq [L]$ with $|A|=\delta L$, $(A-A)\cap h_d(\Lambda_d) \subseteq \{0\}$, $C_0,\delta^{-1}\leq \cQ'$, $q_0 \mid d$, and $d/q_0\leq \cQ$.  Further, let $\eta=c_0\delta$, $Q=\eta^{-2}$, and $Y=\eta^{-2k}$.  As $h_d(H) \subseteq [-L/9,L/9]\setminus \{0\}$, we have
\begin{align*} \sum_{\substack{x \in \Z \\ \bsn\in H}} f_A(x)f_A(x+h_d(\bsn))\nu_d(\bsn)&=\sum_{\substack{x \in \Z \\ \bsn \in H}} 1_A(x)1_A(x+h_d(\bsn))\nu_d(\bsn) -\delta\sum_{\substack{x \in \Z \\ \bsn\in H}} 1_A(x)1_{[L]}(x+h_d(\bsn))\nu_d(\bsn) \\ &\qquad -\delta \sum_{\substack{x \in \Z \\ \bsn\in H}} 1_{A}(x+h_d(\bsn))1_{[L]}(x)\nu_d(\bsn)+\delta^2\sum_{\substack{x \in \Z \\ \bsn\in H}} 1_{[L]}(x)1_{[L]}(x+h_d(\bsn))\nu_d(\bsn)  \\&\leq \Big(\delta^2L -2\delta|A\cap (L/9,8L/9)|\Big)T. 
\end{align*}
Therefore, if $|A \cap (L/9, 8L/9)| \geq 3\delta L/4$, we have
\begin{equation}\label{neg} \sum_{\substack{x \in \Z \\ \bsn\in H}} f_A(x)f_A(x+h_d(\bsn)) \leq -\delta^2LT/2.
\end{equation} 
We see from (\ref{Zest}) and orthogonality of characters that 
\begin{equation}\label{orth} 
\sum_{\substack{x \in \Z \\ \bsn\in H}} f_A(x)f_A(x+h_d(\bsn))=\int_0^1 |\widehat{f_A}(\alpha)|^2 S(\alpha)d\alpha +O_h(LM^{\ell-1}\log^{\l}(dM)),
\end{equation} 
where 
\begin{equation*}S(\alpha)= \sum_{\bsn \in [M]^{\ell} \cap W_{d}(Y)}\nu_d(\bsn)e(h_d(\bsn)\alpha).
\end{equation*} 
Combining (\ref{neg}) and (\ref{orth}), we have  
\begin{equation} \label{mass} 
\int_0^1 |\widehat{f_A}(\alpha)|^2|S(\alpha)|d\alpha \geq \delta^2LT/4.
\end{equation} 
Letting $\gamma=\eta^{-2k}/L$, we deduce in Section \ref{expest} that for $\alpha \in \mathbf{M}_q(\gamma), \ q\leq Q $, we have 
\begin{equation} \label{SmajII} |S(\alpha)| \ll_{h} \begin{cases} (q/\vphi(q))^C((k-1)^2+2)^{\omega(q)}T/q & \l=2 \\  C^{\omega(q)}T/q^{3/2}  & \l\geq 3 \end{cases},
\end{equation} 
where $C=C(h)$, while for $\alpha \in \mathfrak{m}(\gamma,Q)$ we have 
\begin{equation} \label{SminII} |S(\alpha)| \leq  \delta T/8.
\end{equation} From (\ref{SminII}) and Plancherel's Identity, we have \begin{equation*}  \int_{\mathfrak{m}(\gamma,Q)} |\widehat{f_A}(\alpha)|^2|S(\alpha)|d\alpha \leq \delta^2LT/8, \end{equation*} which together with (\ref{mass}) yields \begin{equation}\label{majmass}  \int_{\mathfrak{M}(\gamma,Q)}|\widehat{f_A}(\alpha)|^2|S(\alpha)|d\alpha \geq \delta^2 LT/8. \end{equation} 
From (\ref{SmajII}) and (\ref{majmass}) , we have 
\begin{equation} \label{majmassII} \sum_{q=1}^Q  (q/\vphi(q))^C((k-1)^2+2)^{\omega(q)}q^{-1} \int_{\mathbf{M}_q(\gamma)}|\widehat{f_A}(\alpha)|^2 {d}\alpha \gg_{h} \delta^2L
\end{equation}
when $\l=2$ and 
\begin{equation} \label{majmass3} \sum_{q=1}^Q  C^{\omega(q)}q^{-3/2} \int_{\mathbf{M}_q(\gamma)}|\widehat{f_A}(\alpha)|^2 {d}\alpha \gg_{h} \delta^2L
\end{equation}
when $\l\geq 3$. For $\ell=2$, the function $b(q)=(q/\vphi(q))^C((k-1)^2+2)^{\omega(q)}$ satisfies $b(qr)\geq b(r)$, and we make use of the following proposition, which is based on a trick that originated in \cite{Ruz}.
\begin{proposition}[Proposition 5.6, \cite{ricemax}] \label{rstrick} For any $\gamma,Q>0$ satisfying $2\gamma Q^2<1$ and any function $b: \N \to [0,\infty)$ satisfying $b(qr)\geq b(r)$ for all $q,r\in \N$, we have $$\max_{q\leq Q} \int_{\mathbf{M}'_q(\gamma)}|\widehat{f_A}(\alpha)|^2 {d}\alpha \geq Q \Big(2\sum_{q=1}^Q b(q)\Big)^{-1} \sum_{r=1}^Q \frac{b(r)}{r}\int_{\mathbf{M}_r(\gamma)}|\widehat{f_A}(\alpha)|^2 {d}\alpha. $$
\end{proposition}

\noindent Because $b$ is multiplicative, $b(p^v)=((k-1)^2+2)(1+1/(p-1))^C\ll_k 1$ for all prime powers $p^v$, and $$\sum_{q=1}^Q \frac{b(q)}{q}\leq \prod_{p\leq Q} \left(1+\frac{b(p)}{p}+\frac{b(p)}{p^2}+\cdots\right)=\prod_{p\leq Q} \left(1+\frac{(k-1)^2+2}{p}+O_k(1/p^2)\right)\ll_k \log^{(k-1)^2+2}Q, $$ it follows from \cite[Theorem 01]{HallTen} that 
\begin{equation*} \sum_{q=1}^Q b(q) \ll_k Q\log^{(k-1)^2+1} Q, 
\end{equation*}
and the lemma for $\ell=2$ follows from (\ref{majmassII}) and Proposition \ref{rstrick}. For $\l \geq 3$, since $C^{\omega(q)}\ll_{h,\epsilon} q^{\epsilon}$ for every $\epsilon>0$, the sum $\sum_{q=1}^{\infty}  C^{\omega(q)}q^{-3/2}$ is convergent, and hence (\ref{majmass3}) immediately yields $$\max_{q\leq Q} \int_{\mathbf{M}_q(\gamma)}|\widehat{f_A}(\alpha)|^2 {d}\alpha \gg_h \delta^2 L.$$ Since $\mathbf{M}_q(\gamma)\subseteq \mathbf{M}'_q(\gamma), $ this establishes the lemma for $\ell \geq 3$. \end{proof}

The major contribution of \cite{BloomMaynard} was to replace an iterative estimate on repeated sumsets of rational numbers, developed in \cite{PSS}, with a single higher additive energy estimate. The definitions and results that we import from \cite{BloomMaynard} are as follows, after which we prove Lemma \ref{L2I} in the case $\l=1$.
 
\begin{definition} For $m\geq 1$, $\mathcal{B}\subseteq \T$, and $\vepsilon>0$, we define $$E_{2m}(\mathcal{B}) = |\{b_1,\dots,b_{2m}\in \mathcal{B} : b_1+\cdots+b_m = b_{m+1}+\cdots+b_{2m}\}| $$ and $$E_{2m}(\mathcal{B},\vepsilon)=|\{b_1,\dots,b_{2m}\in \mathcal{B} : \norm{b_1+\cdots+b_m-b_{m+1}-\cdots-b_{2m}}_{\T}\leq \vepsilon\}| .$$
\end{definition}

\begin{lemma}[Theorem 2, \cite{BloomMaynard}]\label{BME} Suppose $m\geq 2$, $Q\geq 4$ and $\mathcal{B}\subseteq \{a/q\in \T: q\leq Q\}$, and for each $q$ let $\mathcal{B}_q$ denote the elements of $\mathcal{B}$ of reduced denominator $q$. If $|\mathcal{B}_q|\leq n$ for all $q$, then \begin{equation*}E_{2m}(\mathcal{B})\leq (Qn)^m\log^{C^m}Q, \end{equation*} where $C>0$ is an absolute constant.
\end{lemma}

\begin{lemma}[Lemma 7, \cite{BloomMaynard}] \label{changlem} Suppose $\vepsilon>0$, $A\subseteq [L]$ with $|A|=\delta L$ and $\mathcal{B}\subseteq \T$. Then, for each $m\geq 1$, \begin{equation*} \sum_{\alpha \in \mathcal{B}} |\widehat{1_A}(\alpha)| \ll \delta^{1-1/2m}LE_{2m}(\mathcal{B},(2L)^{-1})^{1/2m}. \end{equation*}
\end{lemma}

\begin{proof}[Proof of Lemma \ref{L2I} for $\l=1$] Here we follow closely the methods of Lemmas 5 and 6 in \cite{BloomMaynard}. Suppose $\l=1$, $A\subseteq [L]$ with $|A|=\delta L$, $(A-A)\cap h_d(\Lambda_d) \subseteq \{0\}$, $C_0,\delta^{-1}\leq \cQ'$, $q_0 \mid d$, and $d/q_0\leq \cQ$.  Further, let $\eta=c_0\delta$, let $Q=\eta^{-2}$, and let $Y=\eta^{-2k}$. Similar to the beginning of the proof in the $\l\geq 2$ case, but simpler because we use only one balanced function instead of two, we have that if $|A \cap (L/9, 8L/9)| \geq 3\delta L/4$, then 
\begin{equation*}\sum_{\substack{x \in \Z \\ n\in H}} f_A(x)f_A(x+h_d(n))\nu_d(n) = \int_0^1\widehat{f_A}(\alpha)\bar{\widehat{1_A}(\alpha)} S(\alpha) d\alpha + O_h(L(dM)^{k-1}\log(dM)) \leq -3\delta^2LT/4,
\end{equation*}
where $S(\alpha)$ is defined as before, hence \begin{equation*} \int_0^1|\widehat{f_A}(\alpha)||\widehat{1_A}(\alpha)||S(\alpha)| d\alpha  \geq \delta^2LT/2. \end{equation*}
Our deduction of \eqref{SminII} in Section \ref{expest} still applies when $\l=1$, so as before, with Cauchy-Schwarz in place of Plancherel, we have \begin{equation*} \int_{\mathfrak{M}(\gamma,Q)}|\widehat{f_A}(\alpha)||\widehat{1_A}(\alpha)||S(\alpha)| d\alpha = \sum_{q=1}^Q\sum_{(a,q)=1} \int_{\mathbf{M}_q(\gamma)} |\widehat{f_A}(\alpha)||\widehat{1_A}(\alpha)||S(\alpha)| d\alpha \geq \delta^2LT/4. \end{equation*}
We will show in Section \ref{expest} that \begin{equation} \label{1varest} \int_{\mathbf{M}_{a/q}} |S(\alpha)|^2 d\alpha \ll_h C^{\omega(q)}\frac{T^2}{qL}, \end{equation}
for $q\leq Q$ and $(a,q)=1$, where $C=C(h)$, so again applying Cauchy-Schwarz we have \begin{equation*}\sum_{q=1}^Q\sum_{(a,q)=1} C^{\omega(q)}q^{-1/2} \sup_{\alpha \in \mathbf{M}_{a/q}}|\widehat{1_A}(\alpha)| \left( \int_{\mathbf{M}_{a/q}}|\widehat{f_A}(\alpha)|^2 d\alpha\right)^{1/2}\gg_h \delta^2L^{3/2}.\end{equation*}
Let $$R=\left\{a/q\in \T: q\leq Q, \int_{\mathbf{M}_{a/q}(\gamma)}|\widehat{f_A}(\alpha)|^2 {d}\alpha \leq L/Q^5\right\},$$ so $|R|\leq Q^2$ and 
$$\sum_{a/q\in R} C^{\omega(q)} q^{-1/2} \sup_{\alpha \in \mathbf{M}_{a/q}}|\widehat{1_A}(\alpha)| \left( \int_{\mathbf{M}_{a/q}}|\widehat{f_A}(\alpha)|^2 d\alpha\right)^{1/2} \ll_h Q^2\delta L^{3/2}/Q^{5/2} = \delta L^{3/2}/Q^{1/2} = \delta\eta L^{3/2},  $$ 
hence \begin{equation}\label{majmassIV} \sum_{q=1}^Q\sum_{a/q\notin R}  C^{\omega(q)}q^{-1/2} \sup_{\alpha \in \mathbf{M}_{a/q}}|\widehat{1_A}(\alpha)| \left( \int_{\mathbf{M}_{a/q}}|\widehat{f_A}(\alpha)|^2 d\alpha\right)^{1/2} \gg_{h} \delta^2L^{3/2}. \end{equation}
Further, because the measure of $\mathbf{M}_{a/q}$ is $Q^k/L$ and $|\widehat{f_A}(\alpha)|\ll \delta L$, we know $$\int_{\mathbf{M}_{a,q}(\gamma)}|\widehat{f_A}(\alpha)|^2 d\alpha \ll Q^k\delta^2L. $$
Dyadically pigeonholing in both $q$ and the integral value, there exist  $1\leq Q' \leq Q$, $Q^{-k}\leq K \leq Q^{3}$, and $\mathcal{B} \subseteq \{a/q\in \T: q\leq Q, a/q\notin R\}$ such that all reduced denominators in $\mathcal{B}$ are between $Q'$ and $2Q'$, \begin{equation}\label{K2L} \delta^2L/K^2 \leq \int_{\mathbf{M}_{a,q}(\gamma)}|\widehat{f_A}(\alpha)|^2 d\alpha \leq 2\delta^2L/K^2 \quad \text{for all }a/q\in \mathcal{B}, \end{equation} 
and 
\begin{equation}\label{BMsum} \sum_{a/q \in \mathcal{B}}  C^{\omega(q)}q^{-1/2} \sup_{\alpha \in \mathbf{M}_{a/q}}|\widehat{1_A}(\alpha)| \left( \int_{\mathbf{M}_{a/q}}|\widehat{f_A}(\alpha)|^2 d\alpha\right)^{1/2} \gg_{h} \delta^2L^{3/2}/\log^2 Q. \end{equation}  
Letting $\alpha_{a/q}$ denote the point in $\mathbf{M}_{a/q}$ on which $|\widehat{1_A}|$ attains its maximum, substituting \eqref{K2L} into \eqref{BMsum} gives 
\begin{equation}\label{Bmax} \sum_{a/q\in \mathcal{B}} |\widehat{1_A}(\alpha_{a/q})| \gg_h \frac{\delta\sqrt{Q'}KL}{\tau\log^2 Q},  \end{equation} 
where $\tau = \max_{q\leq Q} C^{\omega(q)} \leq \exp(C\log(2\delta^{-1})/\log\log(3\delta^{-1}))$.
Let $$\theta = (\delta^2 L)^{-1}\max_{q\leq Q}\int_{\mathbf{M}_{q}(\gamma)}|\widehat{f_A}(\alpha)|^2 d\alpha.$$ For fixed $q$, we let $\mathcal{B}_q$ denote the elements of $\mathcal{B}$ with reduced denominator exactly $q$, and \eqref{K2L} yields 
\begin{equation*} |\mathcal{B}_q|\delta^2L/K^2 \leq \sum_{a/q\in \mathcal{B}_q} \int_{\mathbf{M}_{a/q}}|\widehat{f_A}(\alpha)|^2 d\alpha \leq \theta\delta^2L,  \end{equation*}
so in particular \begin{equation}\label{Bqb}|\mathcal{B}_q| \leq \theta K^2.  \end{equation}
Letting $m=2\lceil \log\log(3\delta^{-1})\rceil$, \eqref{Bmax} and the pigeonhole principle ensure the existence of $\mathcal{B'}\subseteq \mathcal{B}$, contained in an interval of length $(8m)^{-1}$ satisfying \begin{equation}\label{Bmax2} \sum_{a/q\in \mathcal{B'}} |\widehat{1_A}(\alpha_{a/q})| \gg_h \frac{\delta\sqrt{Q'}KL}{m\tau\log^2 Q}. \end{equation} Letting $\Gamma = \{\alpha_{a/q}: a/q\in \mathcal{B}'\}$, \eqref{Bmax2} and Lemma \ref{changlem} yield
\begin{equation}\label{Gamma1} \frac{\delta\sqrt{Q'}KL}{\tau m\log^2 Q} \ll_h \delta^{1-1/2m}LE_{2m}(\Gamma,(2L)^{-1})^{1/2m}. \end{equation}
However, the elements inside the norm in the definition of $E_{2m}(\Gamma,(2L)^{-1})$ are all rationals of denominator at most $(Q')^{2m}$, and since $(Q')^{2m}=\mathcal{Q}^{O_h(\log\log L)}\leq L$, such a rational can only be less than $(2L)^{-1}$ in absolute value if it is $0$, meaning  $E_{2m}(\Gamma,(2L)^{-1})=E_{2m}(\Gamma)$. Combining with \eqref{Gamma1}, \eqref{Bqb}, and Lemma \ref{BME}, we have 
\begin{equation*}(Q'\theta K^2)^{m}\log^{C^m}Q\geq E_{2m}(\Gamma) \gg_h \delta \left(\frac{\sqrt{Q'} K}{\tau m\log^{2}Q}\right)^{2m}, \end{equation*} 
which rearranges to
\begin{equation*} \theta \geq \frac{\delta^{1/m}}{m^{2}\tau^2\log^{C^m}Q}.
\end{equation*} Since $Q\ll_h \delta^{-2}$ and $m=2\lceil \log\log(3\delta^{-1}) \rceil$, this yields the desired lower bound $$ \theta \gg_h \exp\left(-C\frac{\log(2\delta^{-1})}{\log\log(3\delta^{-1})} \right).$$
\end{proof}

\section{Criteria for $\P$-Deligne Polynomials}
In this section we establish the sufficient conditions for $\P$-Deligne polynomials enumerated in Theorem \ref{Pcrit}. Most of the statements we make here are analogous to certain statements in Sections 2 and 5 of \cite{DR}, and in those cases we simply mention the corresponding statement in \cite{DR} and that the proof is essentially the same.

We begin with two geometric lemmas. The first is a straightforward consequence of the point-counting estimates for varieties over finite fields due to Lang and Weil; a short proof is provided in \cite[Lemma 5.2]{DR}. We use $V^\ns$ to denote the nonsingular points on a variety $V$.

\begin{lemma} \label{langweillem}
Let $k$, $\ell$, $m$, and $r$ be positive integers, and let $q$ be a prime power. Let $V$ be a (reduced) closed subvariety of $\bbP^\ell$, defined over $\F_q$, of degree $k$ and dimension $r$. Let $m \ge 1$ be the number of geometrically irreducible components of $V$ which are defined over $\F_q$. Then
	\begin{equation}\label{eq:LangWeil}
		|V(\F_q)|,\ |V^\ns(\F_q)| = mq^r + O_{k,\ell,r}(q^{r-1/2}).
	\end{equation}
Moreover, the same is true if we replace $V$ with a closed subvariety $W \subseteq \bbA^\ell$.
\end{lemma}

The second geometric lemma is a slight variation on \cite[Lemma 5.3]{DR}.

\begin{lemma}\label{lem:equiv}
Let $V \subseteq \bbP^\ell$ be a variety (reduced, but not necessarily irreducible) of dimension $r \ge 1$ defined over $\Z$, let $V^\ns$ be the nonsingular locus of $V$, and let $V^\ns_0$ be the Zariski open subset of $V^\ns$ obtained by imposing the conditions $x_i \ne 0$ for all $0 \le i \le \ell$. If $p$ is sufficiently large (with respect to $V$), the following are equivalent:
	\renewcommand{\theenumi}{\alph{enumi}}
	\begin{enumerate} 
	\item $V^\ns_0(\F_p) \ne \emptyset$.
	\item $V^\ns_0(\Z_p) \ne \emptyset$.
	\item At least one of the geometric components of $V$ is defined over $\Z_p$ and is not contained in a coordinate hyperplane $\{x_i=0\}$.
	\end{enumerate}
\end{lemma}

\begin{proof}
The proof is the same as for \cite[Lemma 5.3]{DR}, with one additional observation: in the context of showing that (c) implies (a), if $Z$ is an irreducible component of $V$ defined over $\Z_p$ not contained in a coordinate hyperplane, then the number of elements of $Z(\F_p)$ with at least one coordinate $0$ is at most $O_{k,\l,r}(p^{r-1})$ by Lemma \ref{langweillem}; applying Lemma \ref{langweillem} again, there are plenty of nonsingular points leftover, provided $p$ is sufficiently large.
\end{proof}

In the remainder of this section we prove Theorem~\ref{Pcrit}. 
We will use the following definition, which is modified from \cite[Definition 2.7]{DR}.

\begin{definition} For $\ell \in \N$ and $h\in \Z[x_1,\ldots,x_{\l}]$, we say that $h$ is \textit{smoothly $\P$-intersective} if there exists a choice  $\{\bsz_p\}_{p\in \mathcal{P}}$ of $p$-adic integer roots of $h$ such that $(\bsz_p)_i\not\equiv 0 \ (\text{mod }p)$ for all $1\leq i \leq \l$ and all $p$, and  $m_p=1$ for all but finitely many $p$. 
\end{definition}

Now, part (ii) of Theorem~\ref{Pcrit} is proven with an argument which is identical to that of \cite[Proposition 2.5]{DR}. The next proposition is item (iv) from Theorem~\ref{Pcrit}, and we note that the proof is essentially the same as for \cite[Proposition 2.8]{DR}, but using Lemma~\ref{lem:equiv} in place of \cite[Lemma 5.3]{DR}.

\begin{proposition}Suppose $\ell \ge 2$ and $h \in \Z[x_1,\ldots,x_{\l}]$ is Deligne and $\P$-intersective with $\deg(h)=k\geq 2$. If there exists a choice $\{\bsz_p\}_{p\in \P}$ of $p$-adic integer roots of $h$ satisfying $(z_p)_i \not\equiv 0 \ (\textnormal{mod }p)$ for all $1\leq i \leq \l$ and all $p$, and $m_p\in \{1,k\}$ for all but finitely many $p$, then $h$ is $\P$-Deligne. In particular, if $k=2$ or $h$ is smoothly $\P$-intersective, then $h$ is $\P$-Deligne. 
\end{proposition}

\noindent It remains to show parts (i) and (iii) of Theorem \ref{Pcrit}, both of which are a consequence of the following sufficient condition for smooth $\P$-intersectivity, analogous to \cite[Corollary 5.4]{DR}. Note that the following proposition is precisely part (iii) from Theorem~\ref{Pcrit}.

\begin{proposition} \label{Zbar} Suppose $\ell \ge 2$ and $h \in \Z[x_1,\ldots,x_{\ell}]$ is Deligne and $\P$-intersective, and let $h=g_1\cdots g_n$ be an irreducible factorization of $h$ in $\bar{\Z}[x_1,\dots,x_{\l}]$. If, for all but finitely many $p\in \P$, there exists $1 \le i \le n$ such that $g_i$ has coefficients in $\Z_p$ and $x_j \nmid g_i$ for all $1 \le j \le \l$, then $h$ is smoothly $\P$-intersective, hence $\P$-Deligne.
\end{proposition}

\begin{proof} Suppose $\l\geq 2$ and $h\in \Z[x_1,\dots,x_{\l}]$ satisfies the hypotheses of the proposition. Lemma \ref{lem:equiv} gives us exactly what we need for smooth $\P$-intersectivity---a nonsingular $\Z_p$ point, with all coordinates nonzero modulo $p$, for all but finitely many $p$---except in the following pathological scenario: for infinitely many primes $p$, every irreducible component of the variety $\{h=0\}$ defined over $\Z_p$ is in fact a coordinate hyperplane.
\end{proof}

For $\l \ge 3$, a reducible hypersurface in $\bbP^{\l-1}$ must be singular; thus, a Deligne polynomial in at least three variables must be geometrically irreducible. We therefore obtain the following consequence of Proposition~\ref{Zbar}, proving item (i) from---and thus completing the proof of---Theorem~\ref{Pcrit}.

\begin{corollary} If  $\l\geq 3$ and $h\in \Z[x_1,\dots,x_{\l}]$ is Deligne and $\P$-intersective, then $h$ is smoothly $\P$-intersective, hence $\P$-Deligne.
\end{corollary} 

%
%

\section{Exponential Sum Estimates} \label{expest}

In this section, we import all objects and parameters from Section \ref{L2Isec}, introduced prior to and during the proof of Lemma \ref{L2I}. Further, for $q\in \N$ we let $$W_{d,q}(Y) = \left\{ \bsn\in \Z^{\l}: ((r_d)_i+dn_i,q)=1 \text{ for all }1\leq i \leq \l, \ \grad h_d(\bsn) \not\equiv \bszero \text{ mod } p^{\gamma_{d}(p)} \text{ for all } p\leq Y, p^{\gamma_d(p)}\mid q \right\}, $$
and for $\bss \in [q]^{\l}$ and a prime $p$ with $p^{\gamma_d(p)}\nmid q$, we let $j_{d,q,\bss}(p)$ denote the number of $\l$-tuples of congruence classes $\bsc$ modulo $p^{\gamma_d(p)}$ satisfying $\grad h_d(\bsc) \equiv \bszero  \pmod{p^{\gamma_d(p)}}$, $p\nmid (r_d)_i+dc_i$ for all $1\leq i \leq \l$, and $\bsc \equiv \bss \pmod{p^{\ord_p(q)}}$. We let $\epsilon_{d,q}(p)=0$ if $p\mid dq$ and $1$ otherwise. Finally, we let 
\begin{equation*}w_{d,q}(\bss)=\prod_{\substack{p\leq Y \\ p^{\gamma_d(p)}\nmid q}} \left(1- \frac{j_{d,q,\bss}(p)}{((p-\epsilon_{d,q}(p))p^{\gamma_{d}(p)-\ord_p(q)-1})^{\l}} \right). \end{equation*}

\subsection{Major arc estimate} A minor adaptation of the proof of Proposition \ref{sieve2}, with $t=c\sqrt{\log M}/\log Y$, gives the following estimates. For the remainder of the section, we let $E$ denote an error term of the form $M^{\l}\exp(-c\sqrt{\log M})$ for $\l \geq 2$, and $L\exp(-c\sqrt{\log M})$ for $\l=1$, for a constant $c=c(h)>0$, noting that $E$ can absorb terms of the form $\cQ^{O_h(1)}$. 

\begin{lemma}\label{smodq} Suppose $\l \geq 2$. For $0<x_1,\dots,x_{\l}\leq M$, $q\leq \cQ^{O_h(1)}$ and $\bss \in [q]^{\l}$, we have \begin{equation*}\sum_{\substack{\bsn \in B \\ \bsn \equiv \bss (\textnormal{mod }q)}} \nu_d(\bsn) = \left(\frac{\vphi(d)}{\vphi(qd)} \right)^{\l}\frac{w_{d,q}(\bss)}{w_d}\prod_{i=1}^{\l} \left(x_i-\frac{\chi((r_d)_i)x_i^{\rho}}{\rho d^{1-\rho}} \right)1_{W_{d,q}(Y)}(\bss) + O(E). \end{equation*} 
\end{lemma}

\begin{lemma}\label{smodq1} Suppose $\l =1$. For $0<x\leq M$, $q\leq \cQ^{O_h(1)}$ and $s \in [q]$, we have \begin{equation*}\sum_{\substack{n \in B \\ n \equiv s (\textnormal{mod }q)}} \nu_d(n) = \frac{\vphi(d)}{\vphi(qd)}\frac{w_{d,q}(s)}{w_d}\left(\int_1^x h_d'(t)\left(1-\chi(r_d)(dt)^{\rho-1}\right)dt\right)1_{W_{d,q}(Y)}(s) + O(E). \end{equation*} 
\end{lemma}  
 
\noindent We use Lemmas \ref{smodq} and \ref{smodq1}, together with partial summation, to get asymptotic formulas for $S(\alpha)$ near rationals with small denominator. 

\begin{lemma}[Multivariable Partial Summation, Lemma 7.4 \cite{DR}] \label{mps}

Suppose $\ell\in \N$ and $a:\N^{\ell}\to \C$. Suppose further that $b: \R^{\ell}\to \C$ is $C^{\ell}$. For any $X\geq 1$, we have \begin{align*}\sum_{\bsn \in [X]^{\ell}} a(\bsn)b(\bsn)&= A(X,\dots,X)b(X,\dots,X) \\ &+\sum_{i=1}^{\ell} (-1)^i\sum_{1\leq j_1<\cdots<j_i\leq \ell} \int_{[1,X]^{i}} A(\star)\frac{\partial^i b}{\partial x_{j_1}\cdots \partial x_{j_i}}(\star) \ dx_{j_1}\cdots dx_{j_i},\end{align*} where $$A(x_1,\dots,x_{\ell})=\sum_{\bsn \in [x_1]\times \cdots \times [x_{\ell}]} a(\bsn)$$ and $\star=(X,\dots, x_{j_1},\dots ,x_{j_{i}}, \dots, X),$ with $x_{j_1},\dots,x_{j_i}$ plugged into coordinate positions $j_1,\dots,j_i$ and all other coordinates evaluated at $X$.
\end{lemma}

\noindent For the following two lemmas, let $J$ be the sum of the absolute value of all coefficients of $h_d$.

\begin{lemma}\label{Sasym} Suppose $\l \geq 2$. If $a,q\in \N$, $q\leq \cQ^{O_h(1)}$, and $\alpha=a/q+\beta$, then \begin{align*}S(\alpha) &=\left(\frac{\vphi(d)}{\vphi(qd)} \right)^{\l}w_d^{-1}\mathcal{G}(a,q)\int_{[1,M]^{\l}}\prod_{i=1}^{\l}\left(1-\chi((r_d)_i)(dx_i)^{\rho-1}\right)e(h_d(\bsx)\beta)d\bsx\\& + O_{h}\left(E(1+JM^{k}|\beta|)^{\ell}\right),\end{align*}
where $$\mathcal{G}(a,q)=\sum_{\boldsymbol{s}\in [q]^{\ell}\cap W_{d,q}(Y)}w_{d,q}(\bss)e(h_d(\bss)a/q).$$ 
\end{lemma}

\begin{proof} We begin by noting that for any $a,q \in \N$ and $0\leq x_1,\dots,x_{\ell} \leq M$, letting $B=[x_1]\times\cdots\times [x_{\ell}]$, we have
\begin{align*}\mathcal{S}(x_1,\dots,x_{\ell}):&=\sum_{\bsn \in B}\nu_d(\bsn) e(h_d(\bsn)a/q)\\&=\sum_{\boldsymbol{s}\in [q]^{\ell}} e(h_d(\boldsymbol{s})a/q) \sum_{\substack{\bsn \in B \\ \bsn \equiv \bss \ (\text{mod }q)}} \nu_d(\bsn).
\end{align*}

\noindent  Lemma \ref{smodq} then gives \begin{equation} \label{Tx} \mathcal{S}(x_1,\dots,x_{\ell})=\left( \frac{\vphi(d)}{\vphi(qd)}\right)^{\l}w_d^{-1}\prod_{i=1}^{\l} \left(x_i-\frac{\chi((r_d)_i)x_i^{\rho}}{\rho d^{1-\rho}}\right)\sum_{\boldsymbol{s}\in [q]^{\ell}\cap W_{d,q}(Y)} w_{d,q}(\bss) e(h_d(\boldsymbol{s})a/q)+O(E). \end{equation}
Letting $b(\bsn)=e( h_d(\bsn) \beta)$, we now decompose our sum as $$S(\alpha)=\sum_{\bsn \in [M]^{\ell}}\nu_d(\bsn)e(h_d(\bsn)a/q) b(\bsn)$$ and apply Lemma \ref{mps}, yielding 
\begin{align*} S(\alpha)&=\mathcal{S}(M,\dots,M)b(M,\dots,M) \\ &+\sum_{m=1}^{\ell} (-1)^m\sum_{1\leq j_1<\cdots<j_m\leq \ell} \int_{[1,X]^{m}} \mathcal{S}(\star)\frac{\partial^m b}{\partial x_{j_1}\cdots \partial x_{j_m}}(\star) \ dx_{j_1}\cdots dx_{j_m},
\end{align*} 
where $\star$ is as in Lemma \ref{mps}. Substituting (\ref{Tx}) gives the main term 
\begin{align*}&\left( \frac{\vphi(d)}{\vphi(qd)}\right)^{\l}w_d^{-1}\sum_{\boldsymbol{s}\in [q]^{\ell}\cap W_{d,q}(Y)} w_{d,q}(\bss) e(h_d(\boldsymbol{s})a/q)\Big( \prod_{1\leq i \leq \l} \left(M-\frac{\chi((r_d)_{i})M^{\rho}}{\rho d^{1-\rho}}\right) b(M,\dots,M) \\ +& \sum_{m=1}^{\ell} (-1)^m\sum_{1\leq j_1<\cdots<j_m\leq \ell} \prod_{\substack{1\leq i \leq \l \\ i \neq j_1,\dots,j_m}} \left(M-\frac{\chi((r_d)_{i})M^{\rho}}{\rho d^{1-\rho}}  \right)  \\ \cdot & \int_{[1,M]^{m}} \prod_{1\leq i \leq m} \left(x_i-\frac{\chi((r_d)_{j_i})x_i^{\rho}}{\rho d^{1-\rho}}  \right)\frac{\partial^m b}{\partial x_{j_1}\cdots \partial x_{j_m}}(\star) \ dx_{j_1}\cdots dx_{j_m}\Big). \end{align*} 
By iteratively applying integration by parts, this equals 
\begin{equation*}\left(\frac{\vphi(d)}{\vphi(qd)} \right)^{\l}w_d^{-1}\sum_{\boldsymbol{s}\in [q]^{\ell}\cap W_{d,q}(Y)}w_{d,q}(\bss)e(h_d(\bss)a/q)\int_{[1,M]^{\l}}\prod_{i=1}^{\l}\left(1-\chi((r_d)_i)(dx)^{\rho-1}\right)e(h_d(\bsx)\beta)d\bsx, \end{equation*} 
as desired. The error that results from our substitution of (\ref{Tx}) consists of $2^{\l}$ terms of order $$O_{h}\left(E(1+JM^{k}|\beta|)^{\ell}\right),$$ completing the proof. \end{proof}

\noindent Analogous to the deduction of Lemma \ref{Sasym} from Lemma \ref{smodq}, the following asymptotic formula for $\l=1$ follows from Lemma \ref{smodq1}.

\begin{lemma}\label{1asym} Suppose $\l=1$. If $a,q\in \N$, $q\leq \cQ^{O_h(1)}$, and $\alpha=a/q+\beta$, then \begin{align*}S(\alpha) =\frac{\vphi(d)}{\vphi(qd)} w_d^{-1}\mathcal{G}(a,q)\int_1^M h_d'(x)\left(1-\chi(r_d)(dx)^{\rho-1}\right) e(h_d(x)\beta)dx+ O_{h}\left(E(1+JM^{k}|\beta|)\right).\end{align*} 
\end{lemma}

\subsection{Common divisors and gradient vanishing} \label{cdgv}

Here we collect facts assuring that several quantities depending a priori on $h_d$ can actually be bounded in terms of the original polynomial $h$, which we use implicitly in the remainder of Section \ref{expest}.

\begin{definition} For $g(\bsx)= \sum_{|\boldsymbol{i}|\leq k}a_{\boldsymbol{i}}\bsx^{\boldsymbol{i}}  \in \Z[x_1,\dots,x_{\l}]$, we define $\cont(g)=\gcd(\{a_{\boldsymbol{i}}: |\boldsymbol{i}|>0\}).$
\end{definition}

\begin{proposition}[Proposition 3.6, \cite{DR}] \label{idzero} If  $g(\bsx)= \sum_{|\boldsymbol{i}|\leq k}a_{\boldsymbol{i}}\bsx^{\boldsymbol{i}}  \in \Z[x_1,\dots,x_{\l}]$ is identically zero modulo $q\in \N$, then $q \mid k!\gcd(\{a_{\bsi}\}).$ In particular, if $\grad g$ is identically $\bszero$ modulo $q$, then $q \ll_k \textnormal{cont}(g)$.
\end{proposition} 


\begin{proposition}[Proposition 6.5, \cite{DR}] \label{content} If $g\in \Z[x_1,\dots,x_{\l}]$ is strongly Deligne, then  $\textnormal{cont}(g_d) \ll_g 1. $
\end{proposition}

\begin{lemma}[Corollary 7.3, \cite{DR}] \label{gradcor} If $g\in \F_q[x_1,\dots,x_{\l}]$ is Deligne of degree $k\geq 1$, then $$|\{\bsx \in \F_q^{\l}: \grad g (\bsx) = \bszero|\ll_{k,\l} 1.$$
\end{lemma}

\subsection{Local cancellation} The primary purpose for defining the Deligne condition is the following estimate on multivariate exponential sums over finite fields, due to Deligne in his proof of the Weil conjectures.

\begin{lemma}[Theorem 8.4, \cite{Deligne}] \label{delmain} Suppose $\l \in \N$ and $p\in \P$. If $g\in \mathbb{F}_p[x_1,\dots,x_{\l}]$ is Deligne, then \begin{equation*} \left|\sum_{\bsx\in \mathbb{F}_p^{\l}} e(g(\bsx)/p) \right| \leq (\deg(g)-1)^{\l} p^{\l/2}. \end{equation*}
\end{lemma}

As first done in \cite{ricemax}, to effectively utilize \ref{delmain}, we must reduce to the case of prime moduli, which we do with a multivariable version of Hensel's Lemma that follows from  \cite[Theorem 1.1]{Conrad}.

\begin{lemma}[Multivariable Hensel's Lemma] \label{hensel} Suppose $\ell\in \N$, $g\in \Z[x_1,\dots,x_{\ell}]$, $p$ is prime, $\bsn \in \Z^{\l}$, and $\gamma,v\in \N$ with $v\geq 2\gamma-1$. If $g(\bsn)\equiv 0 \ (\textnormal{mod } p^{2\gamma-1})$ and $\grad g (\bsn) \not\equiv \bszero \ (\textnormal{mod } p^{\gamma})$, then there exists $\bsm\in \Z^{\l}$ with $g(\bsm) \equiv 0 \ (\textnormal{mod } p^v$).
\end{lemma}

\noindent Equipped with Lemmas \ref{delmain} and \ref{hensel}, we establish the following estimate on the local exponential sums that appear in our asymptotic formulas.

\begin{lemma}\label{locgen} If $q\in \N$ has prime factorization $q=p_1^{v_1}\cdots p_r^{v_r}$ with $p_1<\cdots< p_t\leq Y < p_{t+1}< \cdots < p_r$, and $(a,q)=1$, then \begin{align*}\left| \mathcal{G}(a,q) \right| &\leq C_1 \prod_{\substack{p\leq Y \\ p \nmid q}} \left(1- \frac{j_{d}(p)}{((p-\epsilon_{d}(p))p^{\gamma_{d}(p)-1})^{\l}} \right) \\ & \cdot \prod_{i=1}^t \left((k-1)^{\l}p_i^{\l/2} +\mathbbm{1}_{\l=1}+(2p-1)\mathbbm{1}_{\l=2}+ [(k-1)p_i^{\l-3/2} + ((k-2)\l+2^{\l}) p_i^{\l-2}]\mathbbm{1}_{\l \geq 3} + j_d(p_i)\right) \\ & \cdot \prod_{i=t+1}^r C_2(v_i+1)^{\ell} p_i^{v_i(\l-1/k)}, \end{align*} where $C_1=C_1(h)$ and $C_2=C_2(h)$. Further, $\mathcal{G}(a,q)=0$ if $v_i\geq 2\gamma_d(p_i)$ for some $1\leq i \leq t$.
\end{lemma}

\begin{proof} Factor $q=p_1^{v_1}\cdots p_r^{v_r}$ as in the lemma. By the Chinese Remainder Theorem, we have \begin{equation*} \mathcal{G}(a,q)= \prod_{\substack{p\leq Y \\ p \nmid q}} \left(1- \frac{j_{d}(p)}{((p-\epsilon_{d}(p))p^{\gamma_{d}(p)-1})^{\l}} \right)\prod_{m=1}^{r} \tilde{\mathcal{G}}(a,p_m^{v_m}), \end{equation*} 
where 
$$ \tilde{\mathcal{G}}(a,p_m^{v_m}) = \sum_{\bss \in [p_m^{v_m}]^{\l}\cap W_{d,p_m^{v_m}}(Y)} e(ah_d(\bss)/p_m^{v_m}) \cdot \begin{cases} c(\bss) & \text{if } m\leq t, \ v_m < \gamma_d(p_m) \\ 1 & \text{else} \end{cases},$$ 
and $|c(\bss)|\leq 1$.

\noindent Suppose $p^v=p_m^{v_m}$ with $\gamma_d(p)>1$ and $v<2\gamma_d(p)$. Since $p^{2\gamma_d(p)-1}\leq p^{3(\gamma_d(p)-1)}$, we can trivially bound the contributions from all such $p^{v}$ by the cube of the product of prime powers $p^{\gamma_d(p)}$ for which $\gamma_d(p)>1$, which is $O_h(1)$, and absorb them into $C_1=C_1(h)$. 

\noindent Next suppose $p^v=p_m^{v_m}$ with $p\leq Y$ and $v=\gamma_d(p)=1$. If $h_d$ is not Deligne modulo $p$, absorb $\tilde{\mathcal{G}}(a,p)$ into $C_1$. Otherwise, recalling that $j_d(p)$ is the number of zeros of $\grad h_d$ modulo $p$ on $J_d(p)$ we have for $p\nmid a$ that  
\begin{equation} \label{expie1} \left|\sum_{\bss \in [p]^{\ell}\cap W_{d,p}(Y)}e(ah_d(\bss)/p)\right| \leq \left|\sum_{\substack{\bss \in [p]^{\ell} \\ p \nmid (r_d)_i+ds_i  \text{ for all } 1\leq i \leq \l}}e(ah_d(\bss)/p) \right|+j_d(p). \end{equation} 
If $p\mid d$, the sum on the right hand side is complete and bounded by $(k-1)^{\l}p^{\l/2}$ by Theorem \ref{delmain}. If $p \nmid d$, let $m_i= -d^{-1}(r_d)_i \ (\text{mod }p)$ for $1\leq i \leq \l$, so the sum on the right hand side above is 
\begin{equation}\label{expie2} \sum_{j=0}^{\l} (-1)^j \sum_{1\leq i_1<\cdots<i_{j}\leq \l} \sum_{\substack{\bss \in [p]^{\ell} \\ s_{i_n} = m_{i_n} \ \text{for all }1\leq n \leq j}}e(ah_d(\bss)/p)  .\end{equation}
For $j=0$, we have a complete sum, which is bounded by $(k-1)^{\l}p^{\l/2}$ by Theorem \ref{delmain}. If $\l=1$, we are removing only a single term, while if $\l = 2$ we are removing $2p-1$ terms. 

\noindent Now suppose $\l\geq 3$. For $j=1$, by geometric irreducibility, we know that $h_d$ is nonconstant modulo $p$ on each of the hyperplanes $x_i=m_i$. Let $g_i(x_1,\dots,x_{i-1},x_{i+1},\dots,x_{\l})=h_d(x_1,\dots,x_{i-1},m_i,x_{i+1},\dots,x_{\l})$ for each $1\leq i \leq \l$. Since we cannot guarantee that $g_i$ is a Deligne polynomial in $\l-1$ variables, we  only use that it is nonconstant and exploit cancellation in a single variable. For each $\tilde{\bss}\in [p]^{\l-2}$, we let $g_{i,\tilde{\bss}}(x)=g_i(x,\tilde{\bss})$ for $i>1$ and $g_{1,\tilde{\bss}}(x)=g_1(\tilde{\bss},x)$. Then, we have  
\begin{align*}\left|\sum_{\substack{\bss \in [p]^{\ell} \\ s_{i} = m_{i} }}e(ah_d(\bss)/p)\right|&= \left|\sum_{\bss \in [p]^{\ell-1}}e(ag_i(\bss)/p)\right|\\ & = \left|\sum_{\tilde{\bss} \in [p]^{\ell-2}}\sum_{x\in [p]}e(ag_{i,\tilde{\bss}}(x)/p)\right| \\& \leq (k-1)p^{1/2}\sum_{\tilde{\bss} \in [p]^{\ell-2}} \gcd(\text{cont}(g_{i,\tilde{\bss}}),p)^{1/2}.
\end{align*}
Consider a coefficient of $g_i$ not divisible by $p$, with a positive exponent on $x_1$ (or $x_{\l}$ if $i=1$, and if no such coefficient exists relabel the coordinates). Then, $\gcd(\text{cont}(g_{i,\tilde{\bss}}),p)=p$ only if $p$ divides the product of the coordinates of $\tilde{\bss}$, which occurs for fewer than $\l p^{\l-3}$ choices. Putting everything together, we have 
\begin{equation*}\left|\sum_{\substack{\bss \in [p]^{\ell} \\ s_{i} = m_{i} }}e(ah_d(\bss)/p)\right| \leq  (k-1)(p^{\l-3/2}+ \l p^{\l-2}),
\end{equation*}
and hence by \eqref{expie1} and \eqref{expie2}, trivially bounding the terms with $j>1$, we know 
\begin{equation*} \left|\sum_{\bss \in [p]^{\ell}\cap W_{d,p}(Y)}e(ah_d(\bss)/p)\right| \leq (k-1)^{\l}p^{\l/2} + (k-1)p^{\l-3/2} + ((k-2)\l+2^{\l}) p^{\l-2} + j_d(p). \end{equation*}
 
\noindent Now suppose that $p^v=p_m^{v_m}$ with $p\leq Y$ and $v\geq 2\gamma_d(p)$, and let $w=2\gamma_d(p)-1$. If $\bss \in [p^v]^{\ell}$ and $\tilde{\bss}$ is the reduced residue class of $\bss$ modulo $p^{w}$, then $h_d(\bss)\equiv p^{w}t+h_d(\tilde{\bss}) \ (\text{mod }p^v)$ for some $0\leq t\leq p^{v-w}-1$. Conversely,  if $\tilde{\bss}\in [p^w]^{\ell}$ with  $\grad h_d(\tilde{\bss})\not\equiv \bszero \  (\text{mod }p^{\gamma_d(p)})$, then for every $0\leq t \leq p^{v-w}-1$, Lemma \ref{hensel} applied to the polynomial $h_d(\bsx)-(p^{w}t+h_d(\tilde{\bss}))$ yields $\bss \in [p^v]^{\ell}$ with $h_d(\bss)\equiv p^{w}t+h_d(\tilde{\bss}) \ (\text{mod }p^v)$.
 
\noindent In other words, the map $F$ on $\Z/p^{v-w}\Z$ defined by $h_d(p^{w}t+\tilde{\bss})\equiv p^{w}F(t)+h_d(\tilde{\bss}) \ (\text{mod }p^v)$ is a bijection.  In particular,  
\begin{align*}\sum_{\bss\in [p^v]^{\l} \cap W_{d,p^v}(Y)}e(ah_d(\bss)/p^v) &= \sum_{\substack{\tilde{\bss}\in [p^w]^{\ell} \\ \grad h_d(\tilde{\bss})\not\equiv \bszero \ (\text{mod }p^{\gamma_d(p)}) \\ p\nmid (r_d)_i+d\tilde{s}_i \text{ for } 1\leq i \leq \l}} \sum_{t=0}^{p^{v-w}-1}e(a h_d(p^{w}t+\tilde{\bss})/p^v)\\  &=\sum_{\substack{\tilde{\bss}\in [p^w]^{\ell} \\ \grad h_d(\tilde{\bss})\not\equiv \bszero \ (\text{mod }p^{\gamma_d(p)})\\ p\nmid (r_d)_i+d\tilde{s}_i \text{ for } 1\leq i \leq \l}} \sum_{t=0}^{p^{v-w}-1}e(a\left(p^{w}t+h_d(\tilde{\bss})\right)/p^v) \\ &=0,\end{align*} 
where the last equality is the fact that the sum in $t$ runs over the full collection of $p^{v-w}$-th roots of unity. 

\noindent Finally, suppose $p^v=p_m^{v_m}$ with $p> Y$, so there is no longer a gradient nonvanishing condition. Similar to the $j=1$ case of \eqref{expie2}, we only exploit cancellation in a single variable. 

\noindent To this end, for  $\tilde{\bss} =(s_2,\dots,s_{\l})\in [p^v]^{\ell-1}$, we define $g_{\tilde{\bss}}$ by $g_{\tilde{\bss}}(x)=h_d(x,\tilde{\bss})$, and we have   
\begin{align*}\left|\sum_{\substack{\bss\in[p^v]^{\ell} \\ p\nmid (r_d)_i+d\tilde{s}_i \text{ for } 1\leq i \leq \l}}e(ah_d(\bss)/p^v)\right| &\leq  \sum_{\tilde{\bss}\in [p^v]^{\ell-1}} \left| \sum_{\substack{x \in [p^v] \\ p\nmid (r_d)_1+dx}} e(ag_{\tilde{\bss}}(x)/p^v) \right| \\ &= \sum_{\tilde{\bss}\in [p^v]^{\ell-1}} \left| \sum_{x \in [p^v]} e(ag_{\tilde{\bss}}(x)/p^v) - \sum_{y \in [p^{v-1}]} e(a\bar{g_{\tilde{\bss}}}(y)/p^{v-1})\right|, \end{align*} 
where $m \equiv -d^{-1}(r_d)_1 \ (\text{mod }p)$, $\bar{g_{\tilde{\bss}}}(y)=(g_{\tilde{\bss}}(m+py)-g_{\tilde{\bss}}(m))/p$, and the second inner sum is only present when $p\nmid d$. By the  standard single-variable complete sum estimate (see \cite{Chen} for example), the first inner sum is bounded by $p^{v(1-1/k)}\gcd(\text{cont}(g_{\tilde{\bss}}),p^v)^{1/k}$. Further, the second inner sum is bounded by $p^{(v-1)(1-1/k)}\gcd(\text{cont}(\bar{g_{\tilde{\bss}}}),p^{v-1})^{1/k}$, and since this term is only present when $p\nmid d$, we know in this case that $\gcd(\text{cont}(\bar{g_{\tilde{\bss}}}),p^{v-1})\ll_h p^{k-1}$. In any case, we have 
\begin{equation*}\left|\sum_{\substack{\bss\in[p^v]^{\ell} \\ p\nmid (r_d)_i+d\tilde{s}_i \text{ for } 1\leq i \leq \l}}e(ah_d(\bss)/p^v)\right| \ll_h p^{v(1-1/k)} \sum_{\tilde{s} \in [p^v]^{\l-1}} \gcd(\text{cont}(g_{\tilde{\bss}}),p^v)^{1/k}.
\end{equation*}
Suppose $a_{\bsi}=a_{i_1,\dots,i_{\ell}}$ with $0<|\bsi|\leq k$ is a coefficient of $h_d$, corresponding to $x_1^{i_1}\cdots x_{\l}^{i_{\l}}$, that is not divisible by $p$. Further, assume that $i_1>0$, as if $i_1=0$ then we could just relabel our coordinates. In this case, for each $0\leq w \leq v$, $\gcd(\text{cont}(g_{\tilde{\bss}}),p^v)=p^w$ only if $p^w\mid s_2^{i_2}\cdots s_{\ell}^{i_{\l}}$, so in particular $p^{\lceil w/k \rceil}\mid s_2\cdots s_{\ell}$, which occurs for fewer than $(w+1)^{\ell-1}p^{v(\ell-1)-w/k}$ choices of $\tilde{\bss}$. In particular, 
\begin{align*}\sum_{\tilde{s}\in [p^v]^{\ell-1}}\gcd(\text{cont}(g_{\tilde{\bss}}),p^v)^{1/k}&\leq \sum_{w=0}^{v} (w+1)^{\ell-1}p^{v(\ell-1)-w/k}p^{w/k} \\ &\leq (v+1)^{\ell}p^{v(\ell-1)}. \end{align*}  
The resulting bound on the exponential sum modulo $p^v$ is a constant depending on $h$ times $$p^{v(1-1/k)}(v+1)^{\ell} p^{v(\l-1)}=(v+1)^{\ell}p^{v(\ell-1/k)},$$ as required. Having accounted for all prime power divisors of $q$, the proof is complete.
\end{proof}

\begin{corollary}\label{loccor} If $(a,q)=1$, then $$|\mathcal{G}(a,q)|\ll_h\prod_{\substack{p\leq Y \\ p \nmid q}} \left(1- \frac{j_{d}(p)}{((p-\epsilon_{d}(p))p^{\gamma_{d}(p)-1})^{\l}} \right) \cdot \begin{cases} C^{\omega(q)} q^{1/2} & \l=1, \ q\leq Y \\  \left(q/\vphi(q)\right)^C((k-1)^2+2)^{\omega(q)} q & \l=2, \ q\leq Y \\ C^{\omega(q)}q^{\l-3/2} & \l\geq 3, \ q\leq Y \\ C^{\omega(q)}\tau(q)^{\l} q^{\l-1/k}  & q>Y \end{cases},$$  where $\tau(q)=\sum_{m\mid q} 1$ and $C=C(h)$.
\end{corollary}

\subsection{Minor arc estimate}

When $\alpha$ is not close to a rational with small denominator, we use the following variant of Weyl's inequality, which is a version of \cite[Theorem 4.1]{lipan}. For $k\in \N$, let $K=2^{10k}$.
\begin{lemma}[Lemma 12, \cite{Rice}]\label{weyl3} Suppose $g(x)=a_0+a_1x+\cdots+a_kx^k \in \Z[x]$ with $a_k>0$, $X, d\in \N$, and $r\in \Z$. If $U\geq \log X$,
$a_k \gg |a_{k-1}| + \cdots +|a_0|,$ and $d,|r|,a_k \leq U^k$, then 
\begin{equation*}\sum_{\substack{x\leq X \\ r+dx \in \P}} \log(r+dx)e(g(x)\alpha) \ll \frac{X}{U}+U^BX^{1-c}
\end{equation*} 
for $B=B(k)$ and $c=c(k)>0$, provided $|\alpha -a/q| < q^{-2}$  for some $U^{K} \leq q \leq g(X)/U^{K}$ and $(a,q)=1$.
\end{lemma}
\noindent Using Lemma \ref{weyl3} to exploit cancellation in only one variable, combined with the techniques of the proof of \cite[Lemma 7.12]{DR} to account for the sieve, yields the following.
\begin{corollary} \label{weyl4} If $C\geq 1$ and $|\alpha-a/q|<q^{-2}$, $(a,q)=1$, for some $\cQ^{C} \leq q \leq M^k/\cQ^{C}$, then $$\left|S(\alpha) \right| \ll_h \cQ^{-C/K}M^{\l} + \cQ^{B}M^{\l-c},$$ where $B=B(C,k)$ and $c=c(k)>0$.
\end{corollary} 

\subsection{Proof of \eqref{SmajII} and \eqref{SminII} for $\l \geq 2$} Fixing $\alpha\in \T$, and letting $C=C(h)$ be a sufficiently large constant, the pigeonhole principle guarantees the existence of $1\leq q \leq M^k/\cQ^C$ and $(a,q)=1$ such that $$\left|\alpha-\frac{a}{q} \right|<\frac{\cQ^{C}}{qM^k}. $$ Letting $\beta=\alpha-a/q$, if $q\leq \cQ^{C}$, we have by Lemma \ref{Sasym} that 
\begin{equation}\label{eg1} S(\alpha) =\left(\frac{\vphi(d)}{\vphi(qd)} \right)^{\l}w_d^{-1}\mathcal{G}(a,q)\int_{[1,M]^{\l}}\prod_{i=1}^{\l}\left(1-\chi((r_d)_i)(dx)^{\rho-1}\right)e(h_d(\bsx)\beta)d\bsx + O_{h}(E), \end{equation}
so in particular \begin{equation*}|S(\alpha)|\ll_h \left(\frac{\vphi(d)}{\vphi(qd)} \right)^{\l}w_d^{-1} |\mathcal{G}(a,q)|T, 
\end{equation*}
which combines with Corollary \ref{loccor} to yield \eqref{SmajII} if $q\leq Q$ and $|\beta|<\gamma$ and \eqref{SminII} if $Q\leq q\leq \cQ^{C}$. Further, assuming $x_1$ appears in a degree $k$ term of $h_d$ (if not, relabel coordinates), if $q\leq \cQ^C$ and $|\beta|\geq \gamma$, standard van der Corput estimates (see for example Lemma 2.8 in \cite{vaughan}) give 
$$\left|\int_1^M e(h_d(\bsx)\beta)dx_1\right| \ll_h |J\beta|^{-1/k} \ll \eta M.$$  Integration by parts yields 
\begin{align*}\left|\int_1^M (1-\chi((r_d)_1)(dx_1)^{\rho-1})e(h_d(\bsx)\beta)dx_1\right| &\ll \eta \left(M-\frac{\chi((r_d)_1)(dM)^{\rho}}{d\rho}+2(1-\rho)M\right),
\end{align*} which combines with \eqref{eg1} to yield \eqref{SminII}. 
Finally, again exploiting cancellation in only a single variable, \eqref{SminII} holds by Corollary \ref{weyl4} if $\cQ^{C}\leq q \leq M^k/\cQ^{C}$.  \qed

\noindent The proof of \eqref{SminII} for $\l=1$ is similar to the corresponding proof above, with partial summation when appropriate to account for the derivative weight. What requires some final attention, however, is the estimate \eqref{1varest} for the $L^2$ mass of $S(\alpha)$ over a full major arc when $\l=1$.

\subsection{Proof of \eqref{1varest} for $\l=1$} Suppose $q\leq Q$ and $\alpha=a/q+\beta$ with $(a,q)=1$ and $|\beta|<\gamma$. By Lemma \ref{1asym}, we have 
\begin{equation*}S(\alpha) =\frac{\vphi(d)}{\vphi(qd)} w_d^{-1}\mathcal{G}(a,q)\int_1^M h_d'(x)\left(1-\chi(r_d)(dx)^{\rho-1}\right) e(h_d(x)\beta)dx+ O_{h}\left(E\right).  
\end{equation*} Further, we see that
\begin{equation*}\left|\int_1^M h_d'(x)e(h_d(x)\beta) \right| = \left|\int_{h_d(1)}^{h_d(M)} e(y\beta) dy \right| \ll |\beta|^{-1}.
\end{equation*} Integration by parts then yields
\begin{equation*} |S(\alpha)|\ll_h \frac{\vphi(d)}{\vphi(qd)} w_d^{-1}|\mathcal{G}(a,q)|T \min\{1,(L|\beta|)^{-1}\}.
\end{equation*} Applying Corollary \ref{loccor} and splitting the integral $$\int_{|\beta|<\gamma} |S(a/q+\beta)|^2 d\beta $$ into the regions $|\beta|<L^{-1}$ and $|\beta|\geq L^{-1}$ gives the required estimate. \qed

\noindent \textbf{Acknowledgements:} The first author's research was partially supported by NSF grant DMS-2302394. The second author's research was partially supported by an AMS-Simons research grant for PUI faculty.\\


\begin{thebibliography}{10}     

\bibitem{Arala}
{\sc N. Arala}, {\em A maximal extension of the Bloom-Maynard bound for sets with no square differences}, preprint (2023), {\tt arXiv:2303.03345}.

\bibitem{BPPS} 
{\sc A. Balog, J. Pelik\'an, J. Pintz, E. Szemer\'edi}, {\em Difference sets without $\kappa$-th powers}, Acta. Math. Hungar. 65 (2) (1994), 165-187.



\bibitem{BloomMaynard}
{\sc T. Bloom, J. Maynard}, {\em A new upper bound for sets with no square differences}, preprint (2020), {\tt arxiv:2011.13266}.

\bibitem{Chen}
{\sc J.R. Chen}, {\em On Professor Hua's estimate of exponential sums}, Sci. Sinica 20 (1977), 711-719.

\bibitem{Conrad}
{\sc K. Conrad}, {\em A multivariable Hensel's lemma}, https://kconrad.math.uconn.edu/blurbs/gradnumthy/multivarhensel.pdf



\bibitem{Deligne}
{\sc P. Deligne}, {\em La conjecture de Weil I}, Pub. Math. I.H.E.S. 43 (1974), 273-307.

\bibitem{DR}
{\sc J. R. Doyle, A. Rice}, {\em Multivariate polynomial values in difference sets}, Discete Analysis, 2021:11, 46pp.



\bibitem{Furst} 
{\sc H. Furstenberg}, {\em Ergodic behavior of diagonal measures and a theorem of {S}zemer\'edi on arithmetic progressions},
  J. d'Analyse Math, 71 (1977), 204-256.
  
\bibitem{Green} 
{\sc B. Green}, {\em On arithmetic structures in dense sets of integers}, Duke Math. Jour. 114 (2002) no.2, 215-238.

\bibitem{GreenNew}
{\sc B. Green}, {\em On S\'ar\"ozy's theorem for shifted primes}, J. Amer. Math. Soc.(2023),  https://doi.org/10.1090/jams/1036.


\bibitem{taoblog}
{\sc B. Green, T. Tao, T. Ziegler}, {\em A Fourier-free proof of the Furstenberg-S\'ark\"ozy theorem}, https://terrytao.wordpress.com/2013/02/28/a-fourier-free-proof-of-the-furstenberg-sarkozy-theorem/.

\bibitem{HallTen}
{\sc R. Hall, G. Tenanbaum}, {\em Divisors}, Cambridge Tracts in Mathematics, vol. 90, Cambridge University Press, 1990.
 
\bibitem{HLR}  
{\sc M. Hamel, N. Lyall, A. Rice}, {\em Improved bounds on S\'ark\"ozy's theorem for quadratic polynomials}, Int. Math. Res. Not. no. 8 (2013), 1761-1782




\bibitem{KMF}
{\sc T. Kamae, M. Mend\`es France}, {\em van der Corput's difference theorem}, Israel J. Math. 31, no. 3-4, (1978), pp. 335-342.

\bibitem{LangWeil} {\sc S. Lang and A. Weil}, {\em Number of points of varieties in finite fields}, Amer. J. Math. 76 (1954),  819-827.

\bibitem{Lewko}
{\sc M. Lewko}, {\em An improved lower bound related to the S\'ark\"ozy-Furstenberg Theorem}, Electron. J. Combin. 22 (2015), No. 32, 1-6.

\bibitem{lipan} 
{\sc H.-Z. Li, H. Pan}, {\em Difference sets and polynomials of prime variables}, Acta. Arith. 138, no. 1 (2009), 25-52.

\bibitem{LM} 
{\sc N. Lyall, \`A. Magyar}, {\em Polynomial configurations in difference sets}, J. Number Theory 129 (2009), 439-450.
 
\bibitem{Lucier}
{\sc J. Lucier}, {\em Intersective sets given by a polynomial}, Acta Arith. 123 (2006), 57-95.

\bibitem{Lucier2}
{\sc J. Lucier}, {\em Difference sets and shifted primes}, Acta. Math. Hungar. 120 (2008), 79-102.

\bibitem{Lyall}
{\sc N. Lyall}, {\em A new proof of S\'ark\"ozy's theorem}, Proc. Amer. Math. Soc. 141 (2013), 2253-2264.

\bibitem{MV} 
{\sc H. L. Montgomery, R. C. Vaughan}, {\em Multiplicative Number Theory I. Classical Theory}, Cambridge Studies in Advanced Mathematics 97, 2007.


\bibitem{PSS}
{\sc J. Pintz, W. L. Steiger, E. Szemer\'edi}, {\em On sets of natural numbers whose difference set contains no squares}, J. London Math. Soc. 37 (1988),  219-231.


\bibitem{ricemax} {\sc A. Rice}, {\em A maximal extension of the best-known bounds for the Furstenberg-S\'ark\"ozy Theorem}, Acta Arith. 187 (2019), 1-41.

\bibitem{thesis}  
{\sc A. Rice}, {\em Improvements and extensions of two theorems of S\'ark\"ozy}, Ph.D. thesis, University of Georgia, 2012. http://alexricemath.com/wp-content/uploads/2013/06/AlexThesis.pdf. 

\bibitem{Rice} 
{\sc A. Rice}, {\em S\'ark\"ozy's theorem for $\P$-intersective polynomials}, Acta Arith. 157 (2013), no. 1, 69-89.

\bibitem{Ricebin}
{\sc A. Rice}, {\em Binary quadratic forms in difference sets}, Combinatorial and Additive Number Theory III, Springer Proc. of Math. and Stat., vol. 297 (2020), 175-196.

\bibitem{Ruz}
{\sc I. Ruzsa, T. Sanders}, {\em Difference sets and the primes}, Acta. Arith. 131, no. 3 (2008), 281-301.

\bibitem{Ruz2}
{\sc I. Ruzsa}, {\em Difference sets without squares}, Period. Math. Hungar. 15 (1984), 205-209.

\bibitem{Ruz3}
{\sc I. Ruzsa}, {\em On measures on intersectivity}, Acta Math. Hungar. 43(3-4) (1984), 335-340. 
 
\bibitem{Sark1}
{\sc A. S\'ark\"ozy}, {\em On difference sets of sequences of integers I}, Acta. Math. Hungar. 31(1-2) (1978), 125-149.

\bibitem{Sark3}
{\sc A. S\'ark\"ozy}, {\em On difference sets of sequences of integers III}, Acta. Math. Hungar. 31(3-4) (1978), 355-386.
 

\bibitem{Slip} {\sc S. Slijep\v{c}evi\'c}, {\em A polynomial S\'ark\"ozy-Furstenberg theorem with upper bounds}, Acta Math. Hungar. 98 (2003),  275-280.



\bibitem{vaughan}
{\sc R. C. Vaughan}, {\em The Hardy-Littlewood method}, Cambridge University Press, Second Edition, 1997.

\bibitem{mwang}
{\sc M. Wang}, {\em A quantitative bound on Furstenberg-S\'ark\"ozy patterns with shifted prime power common differences in primes}, preprint (2021), {\tt arXiv:2102.11441}.

\bibitem{wang}
{\sc R. Wang}, {\em On a theorem of S\'ark\"ozy for difference sets and shifted primes}, Journal of Number Theory, Volume 211 (2020), 220-234.


\bibitem{younis}
{\sc K. Younis}, {\em Lower bounds in the polynomial Szemer\'edi theorem}, preprint (2019), {\tt arXiv:1908.06058}.



  


 



 



 



   












 






 


\end{thebibliography}
\end{document}